\documentclass[12pt]{article} 
\usepackage{amsmath,amssymb,amsfonts,amsthm,amscd,graphicx,psfrag,epsfig}
\usepackage{color}
\definecolor{blue}{rgb}{0,0,0.7}
\definecolor{red}{rgb}{0.75, 0, 0}
\usepackage[titletoc,toc]{appendix}

\newtheorem{theorem}{Theorem}[section]

\newtheorem{theorem-definition}[theorem]{Theorem-Definition}
\newtheorem{theorem-construction}[theorem]{Theorem-Construction}
\newtheorem{lemma-definition}[theorem]{Lemma--Definition}
\newtheorem{lemma-construction}[theorem]{Lemma--Construction}
\newtheorem{lemma}[theorem]{Lemma}
\newtheorem{proposition}[theorem]{Proposition}

\newtheorem{conjecture}[theorem]{Conjecture}
\newtheorem{definition}[theorem]{Definition}

\begin{document}
\newcommand{\Z}{{\mathbb Z}}
\newcommand{\G}{{\rm G}}
\newcommand{\bt}{\begin{theorem}}
\newcommand{\et}{\end{theorem}}
\newcommand{\bl}{\begin{lemma}}
\newcommand{\el}{\end{lemma}}
\newcommand{\bc}{\begin{conjecture}}
\newcommand{\ec}{\end{conjecture}}
\newcommand{\bd}{\begin{definition}}
\newcommand{\ed}{\end{definition}}
\newcommand{\bp}{\begin{proposition}}
\newcommand{\ep}{\end{proposition}}
\newcommand{\be}{\begin{equation}}
\newcommand{\ee}{\end{equation}}
\newcommand{\la}{\label}
\newcommand{\PP}{{\mathbb P}}
\newcommand{\R}{{\mathbb R}}
\newcommand{\Q}{{\mathbb Q}}
\newcommand{\C}{{\mathbb C}}
\newcommand{\lms}{\longmapsto}
\newcommand{\lra}{\longrightarrow}
\newcommand{\hra}{\hookrightarrow}
\newcommand{\ra}{\rightarrow}
\newcommand{\sgn}{\rm sgn}
\newcommand{\bS}{\Bbb S}
\begin{titlepage}
\title{Symplectic double for moduli spaces of $\G$-local systems on surfaces}   
\author{V.V. Fock, A.B. Goncharov}
\end{titlepage}
\date{\it To the memory of Andrei Zelevinsky }
\maketitle

\tableofcontents
\begin{abstract}\begin{footnotesize}
A {\it decorated surface} ${\bS}$  is 
a topological oriented surface with punctures and holes, equipped with a finite 
set  of special points on the boundaries of holes, considered modulo isotopy. 
Each hole boundary has at least one special point. 

Let $\G$ be a  split semi-simple algebraic group over $\Q$. 
We introduce a moduli space 
${\cal D}_{\G, {\Bbb S}}$, and define a collection of special rational coordinate 
systems on it.

 The moduli space ${\cal D}_{\G, {\Bbb S}}$ is the symplectic double 
of the Poisson moduli space ${\cal X}_{\G, {\Bbb S}}$ of framed $\G$-local systems 
on $S$. Its dimension is ${\rm dim}{\cal D}_{\G, {\Bbb S}}= 2{\rm dim}{\cal X}_{\G, {\Bbb S}}$. 
Its symplectic form is upgraded to a 
$K_2$-symplectic structure for which 
the special coordinates are $K_2$-Darboux coordinates. 

\end{footnotesize}

\end{abstract}

\section{Introduction} 

\paragraph{Dual pairs of moduli spaces related to $\G$-local systems.} We defined in \cite{FG1} a pair of 
moduli spaces ${\cal A}_{\G,{\Bbb S}}$ and ${\cal X}_{\G,{\Bbb S}}$  
closely related to the moduli space of $\G$-local 
systems on ${\Bbb S}$. 
We usually consider the ${\cal X}$-moduli space 
 for the adjoint group $\G$, and the ${\cal A}$-moduli space 
for its universal cover $\widetilde \G$,  using the notation 
$({\cal A}_{\widetilde \G,{\Bbb S}}, {\cal X}_{\G,{\Bbb S}})$ for the dual pair.

Each of the moduli spaces is equipped with a positive atlas, equivariant under the 
action of the mapping class group $\Gamma_\bS$ of $\bS$. This allows to 
define their manifolds of positive real points. 
The space 
 ${\cal X}_{PGL_2,{\Bbb S}}(\R_{>0})$ is identified with 
the  modified Teichm\"uller 
space, 
and  ${\cal A}_{{\Bbb S}L_2,{\Bbb S}}(\R_{>0})$ is identified with  Penner's
 decorated Teichm\"uller space of ${\Bbb S}$. 
Furthermore, 
their points with values in any semifield are well defined. 
The real tropical points of each of the spaces are identified with the appropriate modifications of 
Thurston's  measured laminations.  
For a group $\G$ of higher rank we get 
dual pairs of higher Teichm\"uller and  
lamination spaces.  

The moduli space ${\cal X}_{\G,{\Bbb S}}$ is equipped with a $\Gamma_\bS$-equivariant Poisson structure. 

The moduli space ${\cal A}_{\widetilde \G,{\Bbb S}}$ is equipped with a $\Gamma_\bS$-equivariant 
class in $K_2$. 
Notice that a $K_2$-class on a space ${\cal A}$  gives rise to a closed $2$-form on ${\cal A}$. 

There is a map 
$
p: {\cal A}_{\widetilde \G,{\Bbb S}} \lra {\cal X}_{\G,{\Bbb S}}
$ respecting these structures.  

\vskip 3mm

A variety of properties of these two moduli spaces is best explained by the fact that 
they admit $\Gamma_\bS$-equivariant cluster structures of two different types.

{\it Cluster $K_2$-varieties}, which we called before cluster ${\cal A}$-varieties, are geometric incarnations of  
cluster algebras of Fomin-Zelevinsky \cite{FZI}. 

{\it Cluster  Poisson varieties},  
called in \cite{FG2} cluster ${\cal X}$-varieties, 
are dual geometric objects. 

A dual pair $({\cal A}, {\cal X})$ of cluster varieties can be 
assigned to any quiver. One has a Langlands type involution 
${\cal A} \to {\cal A}^\vee, ~{\cal X} \to {\cal X}^\vee$ on cluster varieties. 
There is a deep duality between cluster varieties ${\cal A}^\vee$ and ${\cal X}$.

The space 
${\cal A}_{\widetilde \G,{\Bbb S}}$ (respectively ${\cal X}_{\G,{\Bbb S}})$ has 
a structure of $\Gamma_\bS$-equivariant cluster $K_2$- (respectively cluster Poisson) variety. 
This  means that there is a $\Gamma_{\Bbb S}$-equivariant 
collection of rational 
coordinate systems  of specific cluster nature on each of the spaces.  Dual pairs of moduli spaces 
$({\cal A}_{\widetilde \G, {{\Bbb S}}}, {\cal X}_{\G,{{\Bbb S}}})$ provide   
interesting examples of dual pairs of cluster varieties  
arising in geometry. 


\paragraph{The cluster symplectic double ${\cal D}$ and algebras of global q-difference operators.} 
In \cite{FG3} we defined a symplectic double of a 
cluster Poisson variety ${\cal X}$, called the {\it cluster symplectic variety ${\cal D}$}. 
Its symplectic form is a $K_2$-symplectic form. 

So one assigns now to a quiver a triple of cluster varieties $({\cal A}, {\cal X}, {\cal D})$:

\begin{itemize}
\item  A cluster $K_2$-variety ${\cal A}$;

\item A cluster Poisson variety ${\cal X}$;

\item A cluster $K_2$-symplectic variety ${\cal D}$ -- the symplectic double of the cluster Poisson variety. 
\end{itemize}

The cluster symplectic double ${\cal D}$ appears in the construction of quantization of 
the cluster Poisson variety ${\cal X}$. The quantised algebra of regular functions 
${\cal O}_q({\cal D})$ on ${\cal D}$ was realized 
in {\it loc. cit.} as the algebra of {\it global $q$-difference operators} 
on the cluster $K_2$-variety ${\cal A}$. 

Let us elaborate on this:  although we do not use 
it  in the paper, this shows how the quantum symplectic double 
appears, motivating the  goals of the paper. 
The reader might  skip this, and jump to  discussion of the moduli space ${\cal D}_{\G, \bS}$ 
on the next page. 

\vskip 3mm
One can think about the cluster symplectic double ${\cal D}$ 
as of a  cluster analog of the cotangent bundle to ${\cal A}$. This analogy 
can be seen as follows. 

The cotangent bundle $T^*M$ of a manifold $M$ is the quasiclassical limit 
of the algebra of differential operators on $M$. Precisely, the algebra of differential operators is filtered 
by the degree of an operator; the associate graded is the algebra of regular functions on $T^*M$. 

We define differential operators on a  $M$ locally. 
Namely, in the algebraic-geometric set-up, an algebraic variety $M$ is given by 
a collection of 
coordinate domains $U_i \subset \R^n$ and polynomial gluing maps $\varphi_{ij}: U_i \to U_j$. 
We define a global polynomial differential operator $D$ on $M$ as a collection 
of operators $\{D_i\}$ on the spaces of  functions 
${\cal O}(U_i)$, generated by infinitesimal translations and multiplication by linear functions in $\R^n$, such that the 
linear maps $\varphi_{ij}^*: {\cal O}(U_j) \to {\cal O}(U_i)$ induced by the gluing maps $\varphi_{ij}$ 
intertwine them: $\varphi_{ij}^*(D_jf) = D_i\varphi_{ij}^*(f)$.  

Similarly to this, a cluster variety ${\cal A}$ is glued from 
split algebraic tori ${\rm T}_i = (\C^*)^n$ via positive birational maps $\varphi_{ij}: {\rm T}_i \to {\rm T}_j$. 
On each cluster coordinate torus ${\rm T}_i$ we consider a non-commutative algebra of 
operators acting on functions,  generated by the following ones: 

\begin{itemize}

\item the operators 
  induced by the shift by $q$ of one of the coordinats: $z_i \lms qz_i$, 
and 

\item the operators of multiplication by the coordinates $z_i$. 

\end{itemize}

It is called the algebra of $q$-difference operators on 
${\rm T}_i$, and denoted by ${\cal D}_q({\rm T}_i)$. 

The crucial new idea  is that instead of the linear maps $\varphi_{ij}^*$ 
induced by the pointwise acting gluing maps $\varphi_{ij}$ we use 
certain linear maps between the spaces of functions, called the  {\it intertwiners}, 
which are not induced by any pointwise maps. 

Namely, we consider the Hilbert space ${\cal H}_i:= L^2(\R^n, \omega)$ 
associated to the real positive part $(\R^*_{>0})^n \subset (\C^*_{>0})^n$ of the coordinate torus ${\rm T}_i$, 
identified with $\R^n$ by the logarithm map, and  equip it with the 
Lebesgue measure $\omega = d\log x_1 ... d\log x_n$, where $x_i = \log z_i$. 

The algebra ${\cal D}_q({\rm T}_i)$ of the $q$-difference operators on the cluster torus 
${\rm T}_i$ acts on ${\cal H}_i$, or rather on a certain Schwarz subspace $ {\cal S}({\cal H}_j)$ of ${\cal H}_i$. 

We define,
 by using the (non-compact) \underline{quantum dilogarithm function}, a unitary operator
$$
{\rm I}_{ij}: {\cal H}_j \lra {\cal H}_i. 
$$
The key point is that there is  a unique 
birational map  of the fields of fractions 
of the algebras of $q$-difference operators on the cluster tori: 
\be \la{keyP1}
\psi^*_{ij}: {\rm Frac}({\cal D}_q({\rm T}_j)) \lra {\rm Frac}({\cal D}_q({\rm T}_i)),
\ee
such that the operator ${\rm I}_{ij}$ has the ``intertwining property'':
\be \la{keyP}
{\rm I}_{ij}(F f) = \psi_{ij}(F) (f), ~~~~ \forall f\in {\cal S}({\cal H}_j), ~~\forall F \in {\rm Frac}({\cal D}_q({\rm T}_j)). 
\ee 

This is similar to the isomorphism of the ring of polynomial differential operators in $\R^n$ 
induced by  the Fourier transform in $\R^n$. 
The intertwiner was defined in the  first arXive version of \cite{FG2}, and elaborated  in 
\cite{FG3}.  The property (\ref{keyP}) requires clarification since  it involves 
fractions of q-difference operators. 

\vskip 3mm
Summarising, the analog of the  algebra of differential operators on a variety is the algebra 
${\cal O}_q({\cal D})$ 
of  global $q$-difference operators 
on the cluster $K_2$-variety ${\cal A}$.  
It consists of q-difference operators $D_i\in {\cal D}_q({\rm T}_i)$ on each cluster torus, related by 
transformations (\ref{keyP1}): $\psi^*_{ij}(D_j) = D_i$. 

We expect that the concept of \underline{global} q-difference operators plays a key role 
in the theory of quantum group and its geometrisation, similar to the role of differential operators 
in the Lie group theory.  In particular, quantum groups can be realised 
by global q-difference operators. 
\vskip 3mm

The $q\to 1$ limit of the algebra ${\cal O}_q({\cal D})$ is the 
algebra of functions on the cluster symplectic double ${\cal D}$.  
The cluster symplectic double is glued from split algebraic tori ${\rm T}_i$ 
by  {\it cluster symplectic ${\cal D}$-transformations}, which are the $q\to 1$ limits of the 
non-commutative birational transformations (\ref{keyP1}). 
It turns out that the  formulas for them coincide with the formulas of Fomin-Zelevinsky 
defining mutations in cluster algebras with principal coefficients \cite{FZIV}. 
It is remarkable  that, although 
our approaches and motivations were different, we arrived, independently, 
to the same formulas. This connection deserves to be better understood. 

Although  the cluster symplectic double ${\cal D}$ is just a $q\to 1$ limit 
of a more fundamental non-commutative quantum object, its classical geometry 
is highly non-trivial. Our goal  is to establish this geometry 
in the case when  the cluster set-up is assigned to a pair $(\G, \bS)$. 

\paragraph{Symplectic double moduli space ${\cal D}_{\G, \bS}$.} 
It is natural to ask whether there is a $\Gamma_{\Bbb S}$-equivariant moduli space ${\cal D}_{\G,{\Bbb S}}$ related to the dual pair 
of moduli spaces
$({\cal A}_{ \widetilde \G,{\Bbb S}}, {\cal X}_{\G,{\Bbb S}})$ the same way as the cluster 
symplectic variety ${\cal D}$ is related to the dual pair $({\cal A}, {\cal X})$. 
In Section \ref{Sec5} we  
introduce such a symplectic moduli space.  
\vskip 3mm

The story goes as follows. Let $\bS_{\cal D}$ be the topological double of $\bS$. It is a 
topological surface obtained by gluing the decorated surface $\bS$ 
with its ``mirror'' $\bS^o$, given by the same surface 
with the opposite orientation, 
along the corresponding boundary components. The marked 
points on $\bS$ match under the gluing with the ones on $\bS^o$, 
and give rise to {\it punctures} on the double. 
So   $\bS_{\cal D}$ is a topological surface with punctures and without boundary,
 equipped with an orientation reversing involution $\sigma$ flipping $\bS$ and $\bS^o$.  
The  moduli space ${\cal D}_{\G, \bS}$ 
is a 
 relative of (a finite cover of) the
 moduli space of $\G$-local systems on the double $\bS_{\cal D}$. 

The main result of this paper is a construction of
 a cluster symplectic double atlas on the space ${\cal D}_{\G,\bS}$. It is 
a $\Gamma_\bS$-equivariant collection of special rational coordinate systems 
on the space ${\cal D}_{\G,\bS}$; 
different coordinate systems 
are related by cluster  symplectic ${\cal D}$-transformations 
for the cluster symplectic double of the space  ${\cal X}_{\G,\bS}$. This is a new 
construction even for $SL_2$. 

In the case when $\bS$ is a compact surface without boundary we have 
$$
{\cal D}_{\G, \bS} = {\rm Loc}_{\G, \bS} \times{\rm Loc}_{\G, \bS^o}. 
$$
Here 
${\rm Loc}_{\G, \bS} $ is the moduli space of $\G$-local systems on $\bS$. 
It is already symplectic.
This is the only case when there are no coordinates on 
${\cal D}_{\G, \bS}$ or ${\rm Loc}_{\G, \bS} $, and  
we have nothing new to say.

In general the definition of the moduli space ${\cal D}_{\G, \bS}$ is somewhat subtle: 
 it contains a closed subvariety whose points do not parametrise any kind of 
local systems 
on $\bS_{\cal D}$.

\paragraph{Key features of the symplectic moduli space ${\cal D}_{\G,\bS}$.} 
There is a 
Poisson map
\begin{equation} \label{tyu}
\pi: {\cal D}_{\G,\bS} \lra {\cal X}_{\G,\bS}\times {\cal X}_{\G,\bS^o},
\end{equation}
and an
 involution $i$ of ${\cal D}_{\G,\bS}$ interchanging the two projections in (\ref{tyu}). 

There is an embedding $j:  {\cal X}_{\G,\bS} \hra {\cal D}_{\G,\bS}$. Its image is Lagrangian.  
 The following diagram, where $\Delta_{\cal X}$ is the diagonal 
in ${\cal X}_{\G,\bS}\times {\cal X}_{\G,\bS^o}$, is commutative: 
\begin{equation} \label{8:40}
\begin{array}{ccc}
{\cal X}_{\G,\bS}&\stackrel{j}{\hra} &{\cal D}_{\G,\bS} \\
\downarrow && \downarrow \pi\\
\Delta_{\cal X}&\stackrel{}{\hra} & {\cal X}_{\G,\bS}\times {\cal X}_{\G,\bS^o}
\end{array}
\end{equation}
The pair   $({\cal X}_{\G,\bS}, {\cal D}_{\G,\bS})$ with the maps $i, j, \pi$ 
has a symplectic groupoid structure \cite{W} 
related to the Poisson space ${\cal X}_{\G,\bS}$, where ${\cal D}_{\G,\bS}$ 
is the space of morphisms, and ${\cal X}_{\G,\bS}$ is the space of objects. 

There is a map respecting the closed 2-forms 
\be \la{MPH}
\varphi: {\cal A}_{\widetilde \G, \bS}\times {\cal A}_{\widetilde \G, \bS^o} \lra 
{\cal D}_{\G,\bS}.
\ee

The symplectic double ${\cal D}_{\G,\bS}$ sits in a commutative $\Gamma_\bS$-equivariant 
diagram: 
\begin{equation} \label{11.5.06.1}
\begin{array}{ccc}
{\cal A}_{\widetilde \G,\bS}\times {\cal A}_{\widetilde \G,\bS^o} && \\
&\searrow \varphi &\\
p \times p^o\downarrow  && {\cal D}_{\G,\bS}\\
&\swarrow \pi &\\
{\cal X}_{\G,\bS}\times {\cal X}_{\G,\bS^o} & & 
\end{array}
\end{equation}

The  existence of a $\Gamma_\bS$-equivariant 
positive atlas on the space ${\cal D}_{\G,\bS}$ implies that there  is  a 
$\Gamma_\bS$-equivariant   symplectic space 
${\cal D}_{\G,\bS}(\R_{>0})$ of its positive points. It is isomorphic to  $\R^{-2\chi(\bS){\rm dim}\G}$.

\vskip 3mm
The cluster ${\cal D}$-coordinates allow to produce a $\ast$-algebra 
${\cal O}_q({\cal D}_{\G, \bS})$ - a non-commutative $q$-deformation of the algebra 
${\cal O}({\cal D}_{\G, \bS})$ of regular functions on the moduli space 
${\cal D}_{\G, \bS}$.

There is a canonical class in $K_2$ of the moduli space 
${\cal D}_{\G, \bS}$ providing the symplectic form. 
It gives rise to a canonical geometric quantisation line bundle with 
connection $({\cal L}, \nabla)$ on the moduli space 
${\cal D}_{\G, \bS}$, whose curvature is the symplectic form. 

\paragraph{${\cal D}$-laminations \cite{A}.} 
Dylan Allegretti in his Yale Thesis \cite{A} defined {\it integral ${\cal D}$-laminations} 
on decorated surfaces and proved the following results. Denote by ${\cal D}_L(\bS; \Z)$ the set 
of integral ${\cal D}$-laminations on $\bS$. Then given an ideal triangulation 
$T$ of $\bS$, there is an isomorphism of sets 
$$
a_T: {\cal D}_L(\bS; \Z) \stackrel{\sim}{\lra} (\Z \times \Z)^{\{\mbox{\rm edges of $T$}\}}. 
$$ 
Given a flip $T\to T'$ of ideal triangulations,  the tropicalisation of the cluster symplectic double 
transformations (given by formulas (\ref{zx1qt})-(\ref{4.28.03.11x})) is a piecewise linear 
map 
$$
\mu^t_{T \to T'}: (\Z \times \Z)^{\{\mbox{\rm edges of $T$}\}} \lra (\Z \times \Z)^{\{\mbox{\rm edges of $T'$}\}}.
$$ 
One proves that it 
intertwines the isomorphisms $a_T$ and $a_{T'}$, providing a commutative diagram:
$$
\begin{array}{ccc}
&&(\Z \times \Z)^{\{\mbox{\rm edges of $T$}\}} \\
&\nearrow a_T&\\
{\cal D}_L(\bS; \Z) &&\downarrow \mu^t_{T \to T'}\\
&\searrow a_{T'}&\\
&&(\Z \times \Z)^{\{\mbox{\rm edges of $T'$}\}}
\end{array}
$$
This just means that there is a 
$\Gamma_\bS$-equivariant isomorphism of sets 
$$
A_{\Z^t}: {\cal D}_L(\bS; \Z) \stackrel{\sim}{\lra} {\cal D}_{PGL_2, \bS}(\Z^t).
$$

Allegretti defined {\it measured ${\cal D}$-laminations on $\bS$}, and proved that 
there is a similar isomorphism 
between the set ${\cal D}_L(\bS; \R)$ of all measured ${\cal D}$-laminations on $\bS$ and the set 
${\cal D}_{PGL_2, \bS}(\R^t)$ of the real tropical points of 
${\cal D}_{PGL_2, \bS}$:
 $$
A_{\R^t}: {\cal D}_L(\bS; \R) \stackrel{\sim}{\lra} {\cal D}_{PGL_2, \bS}(\R^t).
$$

\paragraph{${\cal D}$-laminations and canonical pairings.} 
Each integral ${\cal D}$-lamination $l$ on $\bS$ gives rise to 
a positive rational function on the space ${\cal D}_{PGL_2, \bS}$:
\be \la{AF}
{\Bbb I}_{\cal D}(l) \in {\cal O}({\cal D}_{PGL_2, \bS}).  
\ee

The collection of  functions (\ref{AF}) can be viewed as a canonical 
$\Gamma_{\bS}$-equivariant pairing 
$$
{\Bbb I}_{\cal D}: {\cal D}_L(\bS; \Z) \times  {\cal D}_{PGL_2, \bS}(\R_{>0}) \lra \R.
$$

Its renormalised version gives rise to  a $\Gamma_{\bS}$-equivariant pairing 
$$
{\Bbb I}^\R_{\cal D}: {\cal D}_L(\bS; \R) \times  {\cal D}_{PGL_2, \bS}(\R_{>0}) \lra \R. 
$$

The tropicalisation of this pairing is well defined, providing a $\Gamma_{\bS}$-equivariant map 
$$
{\Bbb I}^t_{\cal D}: {\cal D}_L(\bS; \R) \times  {\cal D}_{PGL_2, \bS}(\R^t) \lra \R. 
$$
It coincides with the natural intersection pairing  on ${\cal D}$-laminations \cite{A}. 

So the pairing ${\Bbb I}_{\cal D}$ looks similar to the canonical pairings introduced 
in \cite[Section 12]{FG1}. Yet the functions 
${\Bbb I}_{\cal D}(l)$ do not form a canonical basis: they are rational rather then regular. 

The functions 
${\Bbb I}_{\cal D}(l)$  have nice properties. In particular, the ones assigned to  
integral ${\cal D}$-laminations without loops  
are expressed via ratios of products of 
 Fomin-Zelevinsky F-polynomials \cite{A}. This suggests that for any cluster symplectic double ${\cal D}$ 
one should have a canonical pairing 
$$
{\Bbb I}_{\cal D}: {\cal D}(\Z^t) \times  {\cal D}(\R_{>0}) \lra \R
$$
It remains to be seen how 
this story fits into the framework of canonical bases. 

\paragraph{Organization of the paper.} 
We define the moduli space ${\cal D}_{\G, \bS}$  and establish its  
first properties in Section \ref{Sec5}. We introduce the coordinates in the $SL_2$ case in Section \ref{Sec2}. 
We made an effort to make the paper accessible to geometers. So a geometrically inclined 
reader can skip Section 2, and go to Section \ref{Sec2} where we
   emphasize 
connection with the geometry and Teichmuller theory. 
We start in Section 
\ref{Sec2.1} from the case when $\bS$ is a disc with $m$ points on the boundary.  
We consider in Section \ref{Sec2.2} the case of a surface with 
holes without marked points (and in fact a  more general set-up of a 
surface with a simple lamination). 
The general case is a mixture of these two.  In Section \ref{Sec3} we 
consider the moduli space ${\cal D}_{\G, \bS}$ for $SL_2$. 

In Section \ref{Sec4} we define the coordinates for any 
group $\G$.

In Section \ref{Sec6} we recall 
the definition of the quantum double and its classical counterpart, 
the symplectic double, borrowing from \cite{FG3}. We start from the quantum story since 
it is the simplest way to get the formulas in the classical setting. 
To simplify the exposition, we restrict 
in Section \ref{Sec6} to the simply-laced case. 
The general case is discussed in {\it loc. cit.}.

\paragraph{Remark.} The first draft of this paper originally appeared as the last Section of the 
arXive preprint  arXiv:math/0702397, later published in \cite{FG3} without that Section.  

\paragraph{Acknowledgments.} 
A.G. was 
supported by the  NSF grants  DMS-1059129 and DMS-1301776. 
The first version of this paper appeared as Section 7 of \cite{FG3}.  A.G. 
is grateful to IHES where this paper has mostly been written 
for the hospitality and support. His work at IHES was supported by the National Science Foundation 
under Grant No. 1002477.  We are grateful to the referee 
for useful comments. 

\section{The symplectic double moduli spaces}\la{Sec5}

In Section \ref{Sec5.1} we introduce a moduli space  
${\cal D}^*_{\G, \bS}$. 
To define the special coordinates it is sufficient to deal 
with it. 
However this space is only an approximation to the right
moduli space ${\cal D}_{\G,\bS}$. One of the reasons is that we should have a natural map 
$\varphi$, see (\ref{MPH}). 
Its cluster analog is a key feature of the cluster symplectic double. 
However, as we see for $\G=SL_2$ in  Section \ref{Sec2}, 
 there is no such a map even for the spaces of positive points:
$$
{\cal A}_{\widetilde \G,S}(\R_{>0})\times {\cal A}_{\widetilde \G,S^o}(\R_{>0}) \not\lra 
{\cal D}^*_{\G,S}(\R_{>0}).
$$

Assuming that $\bS=S$ has no marked points, we define in Section \ref{Sec5.2}
 a  moduli 
space ${\cal D}_{\G,S}$ equipped with  a map (\ref{MPH}). 
Theorem \ref{EMBED} provides a birational isomorphism
$
{\cal D}^*_{\G,S} \stackrel{\sim}{\to} {\cal D}_{\G,S}.
$

\subsection{The moduli space ${\cal D}^*_{\G, \bS}$}  \la{Sec5.1}

\paragraph{Flag variety.} 
The flag variety ${\cal B}$ for $\G$ parametrises all Borel subgroups in $\G$. 
It is  isomorphic 
to $\G/B$ where $B$ is a Borel subgroup of $\G$. 
A $\G$-local system ${\cal L}$ on 
a space gives rise to  
the associated local system of  
flag varieties ${\cal L}_{\cal B}:= {\cal L}\times_\G{\cal B}$.   

\paragraph{Decorated flags.} 
 Let $U$ be a  maximal unipotent subgroup of $\G$. 
The {\it decorated flag variety} ${\cal A}_\G$, 
also known as the principal affine space for $\G$,  
is isomorphic to 
$\G/U$. 
A $\G$-local system ${\cal L}$ gives rise to the associated {\it decorated flag local system ${\cal L}_{\cal A}$}:
$$
{\cal L}_{\cal A}:= 
{\cal L}\times_\G {\cal A}_\G \stackrel{\sim}{=} {\cal L}/U.
$$ 
There is a canonical projection $ 
{\cal L}_{\cal A} \to {\cal L}_{\cal B}$. 

 The {\it configuration space of $n$ decorated flags} is defined by 
$$
{\rm Conf}_n({\cal A}_\G):= 
\G\backslash {\cal A}_\G^n. 
$$  
The Cartan group $H$ of $\G$ 
acts on ${\cal A}_\G$ from the right. So 
the group $H^n$ acts on  ${\rm Conf}_n({\cal A}_\G)$.

When the group $\G$ is simply-connected, 
a positive structure on the space 
${\rm Conf}_n({\cal A}_\G)$ was defined in Section 8 of \cite{FG1}. 
The space ${\rm Conf}_n({\cal A}_\G)$ has a structure of the cluster 
${\cal A}$-variety. The case when $\G=SL_m$ is 
described in Section 10 of {\it loc. cit}.

\paragraph{Twisted $\G$-local systems.} 
Denote by  $s_{ \G}$ a central element in $\G$  
given by the image of $-e$ under a principal embedding $SL_2 \to \G$. 
For example, if $\G = SL_m$, then  $s_{\G} = (-1)^me$. 

Let $T'_S$ be the punctured at the zero section tangent bundle of $S$. 
A {\it twisted $\G$-local system on $S$} is a  
$\G$-local system on $T'_S$ with the monodromy 
$s_{\G}$ along a simple loop around the origin in a tangent space to a point of $S$.

\paragraph{The moduli space ${\cal D}^*_{\G,\bS}$.} Let $\bS$ be an arbitrary 
decorated surface with $k> 0$ holes $h_i$. 
We glue it to its mirror $\bS^o$, matching the corresponding pairs of marked points. We get
 a new topological surface  $\bS'_{\cal D}$, equipped 
with  a set of the glued marked points. 
Deleting them  
we arrive at the topological double 
$\bS_{\cal D}$ of $\bS$:
$$
\bS_{\cal D}:= \bS'_{\cal D}- \{\mbox{\rm glued marked  points}\}.
$$ 
We call the deleted points {\it punctures}. 
  The surface $\bS_{\cal D}$ carries an unbounded  simple lamination $\gamma$, 
called the {\it  neck lamination},   
consisting of boundary components of $\bS$ and $\bS^o$ glued together, 
minus the punctures. It is a union of circles and open segments, whose endpoints 
 contain all punctures of $\bS_{\cal D}$.

Let $\sigma: {\bS}_{\cal D} \to {\bS}_{\cal D}$ be the  
involution interchanging the two halves 
of the double. 
Let ${\rm C}_{\widetilde \G}$ be the center of the group $\widetilde \G$. Let 
us consider a subgroup 
\be \la{subgrd}
\Delta_{\widetilde \G} \subset 
{\rm Hom}(H_1({\bS}_{\cal D}, \Z), 
{\rm C}_{\widetilde \G})
\ee
of all maps $f \in {\rm Hom}(H_1({\bS}_{\cal D}, \Z), {\rm C}_{\widetilde \G})$ 
 invariant under the action of  the involution $\sigma$. 
The group ${\rm Hom}(H_1({\bS}_{\cal D}, \Z), 
{\rm C}_{\widetilde \G})$, and hence its subgroup $\Delta_{\widetilde \G} $,
 act on the twisted $\widetilde \G$-local systems on $\bS_{\cal D}$.

\bd The moduli space ${{\cal D}^*}'_{\G,\bS}$ parametrises pairs $({\cal L}, \beta)$ 
where ${\cal L}$ is a twisted $\widetilde \G$-local 
system on $\bS_{\cal D}$ with unipotent monodromies around the punctures, and 
a framing $\beta$ given by: 

\begin{itemize}

\item  a collection of flat sections 
of the associated flag local system ${\cal L}_{\cal B}$:  near each of the punctures, and 
over each  loop of the 
neck lamination $\gamma$. 

\end{itemize}

The moduli space ${\cal D}^*_{\G,\bS}$ is the quotient of  ${{\cal D}^*}'_{\G,\bS}$ 
by the action of the group $\Delta_{\widetilde \G} $:
$$
{\cal D}^*_{\G,\bS}:= {{\cal D}^*}'_{\G,\bS}/\Delta_{\widetilde \G} .
$$ 
\ed

\subsection{The moduli space ${\cal D}_{\G, S}$ and its first properties} \la{Sec5.2}

Below we assume that the decorated surface $\bS$ has no marked points on the boundary. 
So it is just an oriented surface $S$ with boundary. 

Recall that $\widetilde \G$ denotes a simply-connected split algebraic group over $\Q$, and 
$\G$ is its adjoint group. 
Let us recall the definition of the moduli space 
${\cal A}_{\widetilde \G, S}$ and ${\cal X}_{\G, S}$ . 

\begin{definition} [\cite{FG1}] 
i) A decoration on a twisted $\widetilde \G$-local system ${\cal L}$ on $S$ is 
a locally constant section $\alpha_{\cal L}$ of the restriction 
${\cal L}_{\cal A}|_{\partial S}$ of the 
decorated flag bundle ${\cal L}_{\cal A}$ to the boundary $
\partial S$ of $S$. 
The moduli space ${\cal A}_{ \widetilde \G,S}$ parametrises decorated 
twisted $\G$-local system on $S$. 

ii) The moduli space ${\cal X}_{ \G,S}$ parametrises $\G$-local systems ${\cal L}$ on $S$ 
equipped with a 
framing -- a flat section 
of the associated flag local system ${\cal L}_{\cal B}$ over the 
 boundary. 
\end{definition}

There is a canonical map $p: {\cal A}_{\widetilde \G, S} \to {\cal X}_{\G, S}$, 
obtained by forgetting the 
decoration and pushing a twisted 
$\widetilde \G$-local system to a $\G$-local system on $S$.

A framing $\beta$ on a $\G$-local system ${\cal L}$ on $S$ is the same thing as 
a $H^k_\G$-local subsystem, $k=\pi_0(\gamma)$, given by
 the preimage of 
$\beta$ under the map ${\cal L}_{\cal A}|_{\partial S} \to {\cal L}_{\cal B}|_{\partial S}$: 
\be\la{7.18.14.1}
{\cal F}_\beta \subset {\cal L}_{\cal A}|_{\partial S}.
\ee

\bd\label{2.16.06.2}
i) The moduli space ${\cal D}'_{\G,S}$ parametrises  
{\rm gluing data} $({\cal L}_{\pm}, {\beta}_{\pm}, \alpha)$, where: 

\begin{enumerate}

\item $({\cal L}_{\pm }, \beta_\pm)$ are twisted framed 
 $\widetilde \G$-local systems on $S$ and $S^o$ 
with isomorphic restrictions to the boundaries.

\item  A $H^k_{\widetilde \G}$-equivariant map of local subsystems (\ref{7.18.14.1}) 
assigned to the framings $\beta_+$ and $\beta_-$: 
$$
\alpha: 
{\cal F}_{{\beta_+}} \to {\cal F}_{{\beta_-}}.
$$ 
\end{enumerate}

ii) The moduli space ${\cal D}_{\G,S}$ is the quotient of 
${\cal D}'_{\G,S}$ by the action of the group 
$\Delta_{\widetilde \G}$. 
\ed

\paragraph{Properties of the moduli space ${\cal D}_{\G, S}$.} 
Here are some properties of the moduli space ${\cal D}_{\G, S}$ 
matching similar properties of the cluster symplectic double. 

\paragraph{i)} A point $p \in {\cal X}_{\widetilde \G,S}$ determines its mirror image, 
$p^o \in 
{\cal X}_{\widetilde \G, S^o}$. 
Equipping the pair $(p, p^o)$ with the tautological  gluing data, we arrive at 
a point of  ${\cal D}'_{\G,S}$. 
So we get an embedding
$
j': {\cal X}_{\widetilde \G,S} \hra {\cal D}'_{\G,S}. 
$ 
By the very definition, 
$$
{\cal X}_{\G, S} = {\cal X}_{\widetilde \G, S}/{\rm Hom}(\pi_1( S), {\rm C}_\G), ~~~~
{\cal D}_{\G, S} = {\cal D}'_{ \G, S}/{\rm Hom}(\pi_1( S), {\rm C}_\G).
$$
Therefore the embedding $j'$ provides an embedding
$$
j: {\cal X}_{ \G, S} \hra {\cal D}_{\G, S}. 
$$

Next, there is a natural restriction map 
$$
\pi: {\cal D}_{\G, S} \lra {\cal X}_{\G, S}\times {\cal X}_{\G, S^o}.  
$$

So we arrive at a commutative diagram
similar to (\ref{8:40}): 
$$
\begin{array}{ccc}
{\cal X}_{\G, S}&\stackrel{j}{\hra} &{\cal D}_{\G, S} \\
\downarrow && \downarrow \pi\\
\Delta_{{\cal X}_{\G, S}}&\stackrel{}{\hra} & {\cal X}_{\G, S}\times {\cal X}_{\G, S}^{\rm op}
\end{array}
$$

\paragraph{ii)}  There are canonical maps 
$$
\varphi: {\cal A}_{\widetilde \G,  S} \times {\cal A}_{\widetilde \G,  S^o} \lra {\cal D}'_{\G,  S}\lra {\cal D}_{\G,  S}.
$$
Namely, a pair $$
({\cal L}_+, \alpha_{+}) \in 
{\cal A}_{\widetilde \G,  S}, ~~~ ({\cal L}_-, \alpha_{-})\in 
{\cal A}_{\widetilde \G,  S^o}
$$ 
produces a gluing data 
$({\cal L}_\pm, {\beta}_\pm, \alpha)\in {\cal D}'_{\G,  S}$, where the framings 
$(\beta_{+}, \beta_{-})$ are
the images of the decorations $(\alpha_{+}, \alpha_{-})$, and we set 
$\alpha(\alpha_{+}):=\alpha_{-}$. The second map is the canonical projection.


\paragraph{iii)} Forgetting the framing $\beta$, we get a projection 
to the moduli space of twisted $\widetilde \G$-local system on  
$ S$: 
\begin{equation} \label{9:36}
{\rm pr}: {\cal D}_{\G, S} \lra {\cal L}_{\widetilde \G, S}. 
\end{equation}
It is a Galois cover over the generic point with the Galois group $W^k$,  
where $W$ is the Weyl group of $\G$, and $k=\pi_0(\gamma)$.  
The space ${\cal L}_{\widetilde \G, S}$ is
symplectic. It provides 
a $\Gamma_ S$-invariant symplectic structure on an open part of 
${\cal D}_{\G, S}$.

\paragraph{iv)} Our next goal is the following
\bt \la{EMBED} There is a canonical embedding, which is a birational isomorphism: 
$$
{\cal D}^*_{\G, S}\hra {\cal D}_{\G, S}.
$$
\et
To define it, we need a digration on cutting and gluing of some moduli spaces. 

\subsubsection{The moduli space ${\cal X}_{\G, S; \gamma}$} \la{SecX} 
Let us introduce a moduli space 
related to the pair $( S; \gamma)$, where $\gamma$ is a simple lamination on a 
surface $ S$. 

\begin{definition} \label{9:38}
The moduli space ${\cal X}_{\G,  S; \gamma}$ 
parametrises pairs $({\cal L}, \beta)$ 
where ${\cal L}$ is a twisted framed 
$\G$-local system on $ S$ and  a  framing $\beta$ is 
a flat section of the flag bundle 
${\cal L}_{\cal B}$ over the  curve $\gamma$ and the boundary 
$\partial  S$. 
The pair $({\cal L}, \beta)$ is called a twisted framed 
$\G$-local system on $( S; \gamma)$.
\end{definition}

{\bf Example}. 
When $\gamma$ is empty we recover the moduli space 
${\cal X}_{\G, S}$. 

\vskip 3mm

Let $ S-\gamma$ be the surface obtained by cutting $ S$ along the curve $\gamma$.  
Restricting a framed  $\G$-local system on $( S; \gamma)$ to 
$ S-\gamma$ we get the restriction map 
\begin{equation} \label{9:45}
{\rm Res}: {\cal X}_{\G, S; \gamma} \lra {\cal X}_{\G,  S-\gamma}. 
\end{equation} 
It is not a map onto: the  monodromies 
around the matching  boundary components $\gamma_{\pm, i}$ 
of the surface $ S - \gamma$ must coincide. 
\begin{definition} \label{9:38as}
${\cal X}^{\rm red}_{\G, S - \gamma}$ the subspace of ${\cal X}_{\G, S - \gamma}$ 
determined by the condition that the monodromies 
around the loops $\gamma_{\pm, i}$ coincide for every $i$. 
\end{definition}

The following  result is an
 algebraic-geometric version of the cutting and 
gluing techniques  developed in Section 7 of \cite{FG1}, 
in particular Theorem 7.6 there. Observe that 
in {\it loc. cit.} we established cutting and 
gluing properties of the Teichm\"uller space 
${\cal X}_{\G, S}(\R_{>0})$, while their  algebraic-geometric analog 
requires  consideration of the 
moduli space ${\cal X}_{\G, S; \gamma}$. 

\begin{theorem} \label{11.11.06.2}
Let us assume that the center of $\G$ is trivial. Then the restriction map 
\be \la{res}
{\rm Res}: {\cal X}_{\G,  S; \gamma} \lra {\cal X}^{\rm red}_{\G,  S - \gamma} 
\ee
is a fibration over the generic point of 
${\cal X}^{\rm red}_{\G,  S - \gamma}$ with the structure group $H_\G^{\pi_0(\gamma)}$. 
The space ${\cal X}_{\G,  S; \gamma}$ is rational. 
\end{theorem} 

\begin{proof} 
Follows the proof of Theorem 7.6 in {\it loc. cit.}. 
\end{proof}

The restriction map  (\ref{9:45})  is a Poisson map, which admits the following alternative description. 
The monodromies 
along the connected components of $\gamma$ provide a map 
$$
\mu_{\gamma}: {\cal X}_{\G, S; \gamma} \lra H^{k}, \qquad k= \pi_0(\gamma).
$$ 
The lifts of characters of the torus $H^{k}$ commute under the Poisson  
bracket on  ${\cal X}_{\G, S; \gamma}$. 
They provide a Hamiltonian action of the group 
$H^{k}$ on  ${\cal X}_{\G, S; \gamma}$. 
The corresponding Hamiltonian reduction map is the map (\ref{9:45}).

\paragraph{Proof of Theorem \ref{EMBED}.} 
By assigning to a twisted framed $\widetilde \G$-local system on $S_{\cal D}$ 
its restrictions to $S$ and $S^o$ we get an injective map 
\be \la{ORIM}
{{\cal D}'}^*_{\G, S} 
\hra {\cal D}'_{\G, S}.
\ee
 Taking the quotients by the action of the subgroup $\Delta_{\widetilde \G}$, we get an 
 injective map 
$
{\cal D}^*_{\G, S}\hra {\cal D}_{\G, S}.
$ 

One defines a version ${\cal D}^\sharp_{\G, S}$ 
of the moduli space 
${\cal D}'_{\G, S}$ by replacing in Definition \ref{2.16.06.2} the simply-connected group $\widetilde \G$ 
by any split semi-simple group $\G$. Then one sees the following:

1) Recall that $\G'$ denotes the adjoint group of $\widetilde \G$. 
Then  one has
$$
{\cal D}^\sharp_{\widetilde \G, S} \stackrel{\mbox{Def}}{=} {\cal D}'_{\G, S}, ~~~~
{\cal D}^\sharp_{\G', S} \stackrel{\mbox{Th. \ref{11.11.06.2}}}{\sim} {\cal X}_{\G', S; \gamma}.
$$
Here the second map is a birational isomorphism by Theorem \ref{11.11.06.2}. 

2) The space ${\cal D}^\sharp_{\G', S}$ is the quotient of the space 
${\cal D}'_{\G, S}$ by the subgroup
$$
\Delta^\sharp_{\widetilde \G}\subset {\rm Hom}(\pi_1(S), {\rm C}_{\widetilde \G})\times 
{\rm Hom}(\pi_1(S^o), {\rm C}_{\widetilde \G})
$$
consisting of all pairs of maps $(f, f^o)$ which agree on the 
subgroup $\pi_1(\partial S)\subset \pi_1(S)$. Evidently, 
\be \la{CANISOm}
{\rm Hom}(\pi_1(S_{\cal D}), {\rm C}_{\widetilde \G}) = \Delta^\sharp_{\widetilde \G}.
\ee

3) The space ${\cal X}_{\G', S; \gamma}$ is naturally birationally isomorphic to 
the quotient ${{\cal D}'}^*_{\G, S}/{\rm Hom}(\pi_1(S_{\cal D}), {\rm C}_{\widetilde \G})$.

Using isomorphism (\ref{CANISOm}), the map (\ref{ORIM}) provides  a map 
$$
{{\cal D}'}^*_{\G, S}/{\rm Hom}(\pi_1(S_{\cal D}), {\rm C}_{\widetilde \G}) 
\hra {\cal D}'_{\G, S}/\Delta^\sharp_{\widetilde \G}.
$$
Using 1) and 2) it is interpreted as a map:
$
{\cal X}_{\G', S; \gamma} 
\lra {\cal D}^\sharp_{\G', S} \sim {\cal X}_{\G', S; \gamma}.
$
Since it  is tautologically a  birational isomorphism, the 
map (\ref{ORIM}) is also a birational isomorphism.

\section{Cluster coodinates on the symplectic double for $SL_2$} \la{Sec2}

\subsection{The symplectic double of the space of positive 
configurations of points on $\R\PP^1$} \la{Sec2.1}

A collection of $m$ distinct points $(p_1, ..., p_m)$ on an oriented circle is {\it positive} if 
its order is compatible with the orientation of the circle.  Denote by ${\rm Conf}^+_m({\R\PP^1})$ the moduli space of 
positive configurations of $m$ points on the circle modulo the action of the group $PSL_2(\R)$. 
We denote by ${\rm Conf}^-_m(\R\PP^1)$ a similar space of the {\it negative} configurations of points on the circle -- 
a configuration is negative if reversing its order we get a positive configuration. 

The moduli space ${\rm Conf}^+_m({\R\PP^1})$ parametrises ideal geodesic $m$-gons. Indeed, we identify  
the oriented 
$\R\PP^1$ with the boundary of the oriented hyperbolic disc, and assign to a configuration of points $(p_1, ..., p_m)$ on 
$\R\PP^1$ the ideal geodesic $m$-gon with vertices at these points. 

Pick horocycles $h_1, ..., h_m$ at the vertices $p_1, ..., p_m$ of an ideal geodesic $m$-gon. 
For any two horocycles $h_i, h_j$ there is a number  $l(h_i, h_j)$ -- the distance between them 
along the geodesic  connecting $p_i $ and $p_{j} $.  
Namely, if the discs bounded by the horocycles are disjoint, it is the length of the geodesic segment  
between them. Otherwise it is its negative.

When $m=2n$ is even, there is a map, which we call the {\it Casimir map}:
$$
C: {\rm Conf}^+_{2n}(\R\PP^1) \lra \R.
$$
Namely, set 
\be \la{2.1.12.1}
C(p_1, ..., p_{2n}):= l(h_1, h_2)  - l(h_2, h_3) + \ldots + l(h_{2n-1}, h_{2n})  - l(h_{2n}, h_1). 
\ee
It does not depend on the choices of the horocycles. Alternatively, for every vertex $p_i$ 
there is a map sending a geodesic $p_{i-1}p_{i}$ to the geodesic $p_{i+1}p_{i}$. It assigns to every point $c\in p_{i-1}p_{i}$ the intersection of the horocycle passing through $c$ and centered at $p_i$ with the geodesic $p_{i+1}p_{i}$. The composition of these maps is 
a map of the geodesic $p_1p_2$ to itself preserving the length. It preserves the orientation 
if $m$ is even and reverses it if $m$ is odd. So when $m$ is even, it is a translation by the number (\ref{2.1.12.1}). 
If $m$ is odd, it has a unique 
stable point, providing a preferred collection of horocycles. 

\vskip 3mm

Denote by ${\rm Conf}^\sharp_m(\C\PP^1)$ the moduli space of pairs $\{(z_1, ..., z_m), \alpha\}$, considered  
modulo the action of $PGL_2(\C)$, where $(z_1, ..., z_m)$ is 
a configuration of $m$ distinct points  on $\C\PP^1$, and $\alpha$ is an 
isotopy class of a simple oriented loop 
passing through the points so that their order is compatible with the loop orientation. 

A point of ${\rm Conf}^\sharp_{m}(\C\PP^1)$ can be thought of as 
a sphere with a complete 
hyperbolic structure with $m$ cusps $(p_1, ..., p_m)$, plus a simple geodesic polygon  connecting them. 
Cutting along the 
geodesic polygon isotopic to $\alpha$ we get two ideal geodesic $m$-gons.  
Their vertices provide a positive configuration of points on $\R\PP^1$, 
and a negative one. So we get a map 
$$
{\rm Cut}_m: {\rm Conf}^\sharp_{m}(\C\PP^1) \lra {\rm Conf}^+_{m}(\R\PP^1)\times {\rm Conf}^-_{m}(\R\PP^1).
$$

\bt \label{4.15.11.1}
When $m$ is odd, the map ${\rm Cut}_m$ is an isomorphism. 

When $m$ is even, it is a principal $\R$-fibration over the subspace  given by pairs of configurations of points 
with the same values of the Casimirs.  
\et

{\bf Proof}. Consider two geodesic polygons 
$$
P = (p_1, ..., p_m) ~~\mbox{and}~~ P^o = (p^o_1, ..., p^o_m).
$$ 
Choose horocycles $(h_1 , ..., h_m )$ at the vertices of $P$. 
Choose a horocycle $h_1^o$ at the vertex $p_1^o$ of $P^o$. 
Then there is a unique horocycle $h_2^o$ at the vertex $p_2^o$ such that 
$l(h_1^o, h^o_2) = l(h_1, h_2)$. Similarly there is a unique horocycle $h_3^o$ at the vertex $p_3^o$ 
such that $l(h_2^o, h^o_3) = l(h_2, h_3)$. And so on, till we get to the original point $p_1^o$. 
Here it is an {\it a priori}  non-trivial condition that the horocycle constructed by using the horocycle 
$h_m^o$ coincides with the horocycle $h_1^o$. So, for given horocycles $(h_1 , ..., h_m )$,  we need 
horocycles $(h^o_1 , ..., h^o_m )$ 
satisfying a system of linear equations
$$
l(h_1, h_2) = l(h^o_1, h^o_2), ~~l(h_2, h_3) = l(h^o_2, h^o_3), ~~ \ldots, ~~l(h_m, h_1) = l(h^o_m, h^o_1).
$$
When $m$ is odd, it has a unique solution. When $m$ is even, the solution exists only 
when the values of the Casimirs are equal, and is parametrised by one parameter: there is an action of the group $\R$ on the solutions given by $h_i^o \lms h_i^o+ (-1)^ic $. 

Finally, there is a unique way 
to glue the ideal polygon $P $ with the horocycles $(h_1 , ..., h_m )$ and the ideal polygon $P^o$ 
with the horocycles $(h_1^o, ..., h_m^o)$, matching the horocycles, getting a hyperbolic surface with $m$ punctures.  
The result does not depend on the choice of horocycles $(h_1, ..., h_m)$, as well as the  
horocycle $h_1^o$. In particular, we constructed an inverse map 
$$
{\rm Glue}_{2n+1}: {\rm Conf}^+_{2n+1}(\R\PP^1)\times 
{\rm Conf}^-_{2n+1}(\R\PP^1)\stackrel{\sim}{\lra} {\rm Conf}^\sharp_{2n+1}(\C\PP^1). 
$$
Gluing polygons give a punctured surface (a surface with a complete metric) if and only 
if around each vertex on the glued surface there exists a horocycle.

\begin{figure}[ht]
\centerline{\epsfbox{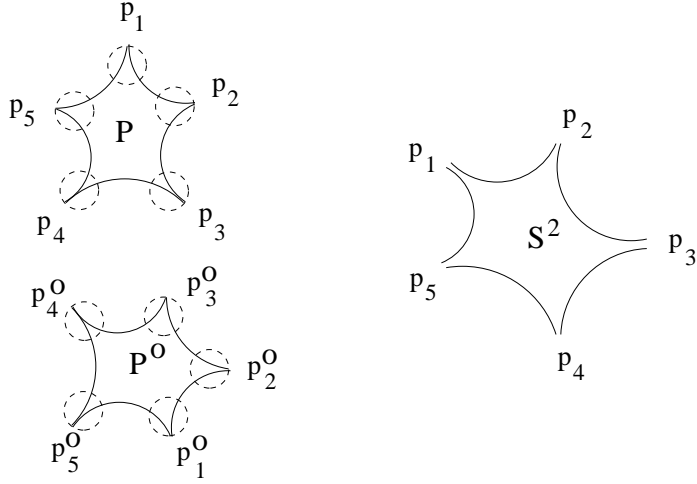}}
\caption{Gluing two ideal geodesic pentagons into a hyperbolic sphere with $5$ cusps.}
\label{double15}
\end{figure}

\subsubsection{Coordinates on the space ${\rm Conf}^\sharp_{m}(\C\PP^1)$}  
\la{Sec2.1.1}
Let $((z_1, ..., z_m), \alpha)$ be a point of  ${\rm Conf}^\sharp_{m}(\C\PP^1)$. 
Let $P_m$ be a convex $m$-gon whose vertices are parametrised by the points $(z_1, ..., z_m)$ 
so that the order of the vertices is compatible with a cyclic order of the points on the loop $\alpha$. 
Let $T$ be a triangulation of $P_m$. 
Let us assign to $T$ a coordinate system $\{x_E, b_E\}$ on the space ${\rm Conf}^\sharp_{m}(\C\PP^1)$, 
where $\{E\}$ are the edges of $T$. The set of edges of $T$ does not include the sides of the polygon. 

We think of a configuration of complex points 
$(z_1, ..., z_m)$ as of a hyperbolic structure on an oriented sphere ${\Bbb S}^2$ with cusps $(p_1, ..., p_m)$. 
Pick horocycles $(h_1, ..., h_m)$ at the cusps. Let $G_\alpha$ be the ideal geodesic polygon 
with vertices at the cusps which is isotopic to  $\alpha$. 
Cutting ${\Bbb S}^2$ along the geodesic polygon $G_\alpha$, we get two ideal geodesic polygons $P$ and $P^o$. 

Consider a geodesic triangulation of $P$ realising the triangulation $T$. An edge $E$ determines an ideal geodesic quadrilateral $(p_a, p_b, p_c, p_d)$ of this triangulation 
with the diagonal $E= p_ap_c$. Let $l(h_i, h_j)$ (respectively $l^o(h_i, h_j)$) 
be the distance between the horocycles $h_i, h_j$ along the geodesic connecting $p_i$ and $p_j$ inside of $P$ 
(respectively $P^o$). 

\begin{definition} \label{4.16.11.1} The coordinates $x_E$ and $b_E$ assigned to the edge $E$ are given by 
$$
x_E:= l(h_a, h_b) - l(h_b, h_c) + l(h_c, h_d) - l(h_d, h_a), \qquad  
b_E := l^o(h_i, h_j) - l(h_i, h_j). 
$$
\end{definition} 
Evidently  they are independent of the choice of the horocycles $h_i$. 
The coordinates 
$\{x_E\}$ are the standard coordinates on the space of ideal geodesic polygons $P$.

\vskip 3mm
The exponents $X_E:= {\rm exp}(x_E)$ and $B_E:= {\rm exp}(b_E)$ 
form a coordinate system $\{X^T_E, B^T_E\}$ assigned to a triangulation $T$. 
We show below that the coordinate systems corresponding to triangulations $T$ of the $m$-gon 
provide a positive real atlas on the space ${\rm Conf}^\sharp_{m}(\C\PP^1)$.

\subsubsection{Geometric interpretation of the cluster symplectic double of type $A_m$} \la{Sec2.1.2}
The definition of cluster symplectic double us recalled in the Appendix. 

In particular, a root system of type $A_m$ gives rise to 
a cluster symplectic double which we denote by ${\cal D}_{A_m}$. 
Let ${\cal D}^+_{A_m}$ be the space of its real positive points. 
The space ${\cal D}_{A_{m}}$ is equipped with a cluster atlas $\{X^T_E, B^T_E\}$ whose 
coordinate systems are parametrised by the triangulations $T$ of a convex $(m+3)$-gon. 
\bt
There is a unique isomorphism
$$
{\rm Conf}^\sharp_{m+3}(\C\PP^1) \stackrel{\sim}{\lra} {\cal D}_{A_m}^+
$$
sending the atlas $\{X^T_E, B^T_E\}$ on ${\rm Conf}^\sharp_{m+3}(\C\PP^1)$ 
to the cluster 
atlas on the symplectic double ${\cal D}^+_{A_m}$. 
\et

{\bf Proof}. Consider an ideal geodesic quadrilateral with vertices $p_1, p_2, p_3, p_4$. 
Denote by $(B_{ij}, X_{ij})$ the pair of coordinates assigned to the edge $p_ip_j$. 
Let us calculate how the coordinate $B_{13}$ changes 
under the flip at the edge $E_{13} = p_1p_3$, see Fig. \ref{double5}. 

We will use shorthands $l_{ij} = l(h_i, h_j)$, $l^o_{ij} = l^o(h_i, h_j)$. Recall the Pl\"ucker relation
$$
{\rm exp}(l_{13}){\rm exp}(l_{24}) = {\rm exp}(l_{12}){\rm exp}(l_{34})  + 
{\rm exp}(l_{14}){\rm exp}(l_{23}). 
$$
There is a similar relation for the ${\rm exp}(l^o_{ij})$. 
Then the  flipped $B$-coordinate 
$B_{24}$ equals
$$
B_{24}:= \frac{{\rm exp}(l^o_{24})}{{\rm exp}(l_{24})} = 
\frac{\Bigl({\rm exp}(l^o_{12} +l^o_{34}) + {\rm exp}(l^o_{14} +l^o_{2,})\Bigr) {\rm exp}(l_{13})}{
\Bigl({\rm exp}(l_{12} +l_{34}) + {\rm exp}(l_{14} +l_{23})\Bigr) {\rm exp}(l^o_{13})} 
= 
\frac{\Bigl(B_{12}B_{34} + X_{13} B_{14}B_{23}\Bigr)}
{\Bigl(1 + X_{13} \Bigr)B_{13}}.
$$
 This agrees with the mutation formula for the $B$-coordinates. 
The mutation formulas for the $X$-coordinates are the standard ones. The theorem is proved. 

\begin{figure}[ht]
\centerline{\epsfbox{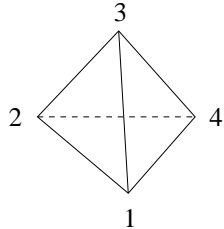}}
\caption{A flip at the edge $E_{13}$.}
\label{double5}
\end{figure}

\paragraph{The symplectic structure \cite{FG3}.} The symplectic structure 
on the space ${\rm Conf}^\sharp_{m}(\C\PP^1)$ 
in the coordinate system $\{b_E, x_E\}$ related to any 
ideal triangulation $T$ of the $m$-gon $P$ is given by 
$$
-\frac{1}{2}\sum_{E, F} \varepsilon_{EF} db_E\wedge db_F - \sum_{E} db_E\wedge dx_E.
$$
The corresponding Poisson bracket is given by 
$$
\{x_E, x_F\} = \varepsilon_{EF} x_Ex_F, \quad \{x_E, b_F\}=\delta_{EF} \quad \{b_E, b_F\}=0.
$$
The space ${\rm Conf}^+_{m}(\R\PP^1)$ is a Poisson space with the Poisson structure 
given in the coordinate system $\{x_E\}$ related to any ideal triangulation $T$ by the formulas
$$
\{x_E, x_F\} = \varepsilon_{EF} x_Ex_F.  
$$
It is non-degenerate if $m$ is odd, providing a symplectic structure on ${\rm Conf}^+_{m}(\R\PP^1)$. 
When $m$ is even, 
the Casimir function $C$ generates the center of the Poisson structure. 

The cutting map ${\rm Cut}_m$ is a Poisson map. Notice that $\varepsilon_{EF}^o= - \varepsilon_{EF}$. 

\subsubsection{Theorem \ref{4.15.11.1} and the Bers double uniformization theorem.} \la{Sec2.1.3} The moduli space 
${\rm Conf}^+_{m}(\R\PP^1)$ is 
the Teichm\"uller space parametrising complex structures on the disc $\widehat D_m$ 
with $m$ marked points on the boundary. 
The moduli space 
${\rm Conf}^\sharp_{m}(\C\PP^1)$ should be viewed as a baby version of the space of quasifuchsian groups for the  
disc $\widehat D_m$, with Theorem \ref{4.15.11.1} being the analog of the Bers double uniformization theorem. 

\vskip 3mm
Let $S$ be a closed hyperbolic surface. Let $\partial_\infty\pi_1(S)$ be the boundary at infinity 
of the fundamental group of $S$. It is a cyclic $\pi_1(S)$-set, homeomorphic to an oriented circle. 
A map $\partial_\infty\pi_1(S) \lra \R\PP^1$ is {\it positive} if it preserves the cyclic order. 
The Teichm\"uller space ${\cal T}(S)$  is identified (\cite[Lemma 1.1]{FG1})  with 
the set of $\pi_1(S)$-equivariant positive maps 
$$
\varphi: \partial_\infty\pi_1(S) \lra \R\PP^1 ~~~\mbox{modulo the action of $PSL_2(\R)$}. 
$$

\bd A representation 
$
\rho: \pi_1(S) \lra PSL_2(\C)
$ 
is {\rm quasifuchsian} if the induced map  
$$
\psi_\rho: \partial_\infty\pi_1(S) \lra \C\PP^1
$$
is a homeomorphism on its image, i.e. its limit set is a Jordan curve. 
\ed

A quasifuchsian representation $\rho$ is uniquely described by the $\pi_1(S)$-equivariant map $\psi_\rho$. 
The condition that $\rho$ is quasifuchsian is 
equivalent to the following condition on the limit set $$
C_\rho:= \psi_\rho(\partial_\infty\pi_1(S))
$$ of a representation $\rho$:
the convex core, defined as the convex hull of $C_\rho$ modulo the action 
of $\pi_1(S)$, is compact, i.e. its projection to ${\cal H}^3/\rho(\pi_1(S))$ is compact.

Denote by $Q(S)$ the space of quasifuchsian representations of $\pi_1(S)$ modulo the conjugation. 
Then cutting $\C\PP^1$ along the limit set $C_\rho$ we get the Bers map
$$
\beta: Q(S) \lra {\cal T}_S \times {\cal T}_{S^0}.
$$
Indeed, let $\C\PP^1 - C_\rho = D \cup D^0$. Then $\pi_1(S)$ acts discretely on  $D$ and $D^o$, providing 
Riemann surfaces $D/\rho(\pi_1(S))$ and $D^0/\rho(\pi_1(S))$ homeomorphic to $S$. 
The disc $D^o$ and the second 
surface have the opposite orientation. 
Here is the Bers double uniformization theorem.
\bt
The Bers map is an isomorphism. 
\et

Unlike the Bers theorem, Theorem \ref{4.15.11.1} has a simple constructive proof. 
Here is an approach for a new  
proof of the Bers double uniformization theorem, as a limit of Theorem \ref{4.15.11.1}.

\vskip 3mm
Pick a finite subset $C_{2n+1} \subset \partial_\infty\pi_1(S)$. Then 
$\psi_\rho(C_{2n+1})$ is a configuration of $2n+1$ points on $\C\PP^1$, 
and the limit curve $C_\rho$ provides a loop $\alpha$. So we get a point of 
${\rm Conf}^\sharp_{2n+1}(\C\PP^1)$. The Bers map is described by two $\pi_1(S)$-equivariant positive maps 
$$
\varphi, \varphi^o: \partial_\infty\pi_1(S) \lra \R\PP^1.
$$
Their restriction to $\psi_\rho(C_{2n+1})$ should converge to the cutting map: 
as the subset $C_{2n+1}$ approximates $\partial_\infty\pi_1(S)$, the vertical arrows 
in the diagram below should approximate isomorphisms
\be
\begin{array}{ccc}
Q(S) &\stackrel{\beta}{\lra} &{\cal T}_S \times {\cal T}_{S^0}\\
\downarrow &&\downarrow \\
{\rm Conf}^\sharp_{2n+1}(\C\PP^1)& \lra&{\rm Conf}^+_{2n+1}(\R\PP^1)\times {\rm Conf}^+_{2n+1}(\R\PP^1)
\end{array}
\ee

\bc
When the subsets $C_{2n+1}$ approximate $\partial_\infty\pi_1(S)$, the maps $\psi_{C_{2n+1}}$ converge to a limit 
$$
\psi: \partial_\infty\pi_1(S) \lra \C\PP^1,
$$
providing a quasifuchsian representation of $\pi_1(S)$. 
\ec

\subsection{Teichm\"uller space for a closed 
surface with a simple lamination} \la{Sec2.2}
Let $\Sigma$ be a closed oriented hyperbolic surface. 
A {\it simple lamination} on  $\Sigma$ is  
a finite  collection $\{\gamma_i\}$  of simple  non-trivial 
disjoint nonisotopic loops on $\Sigma$ modulo isotopy.  
So  $\gamma:= \cup_i\gamma_i$ is a curve without self-intersections. 

Let us introduce a moduli space ${\cal X}^+_{\Sigma; \gamma}$ assigned to 
a simple lamination $\gamma$ on $\Sigma$.\footnote{We show in Section 2.2 that it is the set of $\R_{>0}$-points of 
a moduli space ${\cal X}_{PGL_2, S;\gamma}$ defined there.}
It will differ from $\mathcal{X}^+_\Sigma$ 
by including some nodal surfaces and some discrete data.
It has a stratification parametrised by 
collections of components of  $\gamma$. The open stratum parametrises 
complex structures on $S$ plus a  choise of an orientation for every loop 
of $\gamma$. Let us define the stratum assigned to 
a collection of loops $\{\gamma_1, \ldots , \gamma_k\}$. Let us pinch these loops  
to the nodes $p_1, \ldots , p_k$, getting a singular surface 
$  \Sigma_{p_1, \ldots , p_k}$ with a simple lamination $\gamma_{p_1, \ldots , p_k}$ 
given by the image of  
$\gamma - 
\{\gamma_1, \ldots , \gamma_k\}$. 

We say that a horocycle $c'$ at a node $p$ is obtained from 
a horocycle $c$ at  $p$ by a shift by $l \in \R$ 
if both are on the same side of $p$, 
and the distance from $c$ to $c'$ in the off $p$ direction is $l$.

\begin{definition} The stratum ${\cal X}^+_{S;\gamma; p_1, \ldots , p_k}$
parametrises complex structures on $  \Sigma_{p_1, \ldots , p_k}$ 
plus the following gluing data:

\begin{itemize}
\item An orientation for every loop $\gamma_i$ of the simple lamination 
$\gamma_{p_1, \ldots , p_k}$ on $\Sigma_{p_1, \ldots , p_k}$. 

\item For every node $p_i$, a pair of horocycles $(c_{-, i}, c_{+, i})$ centered at the node $p_i$, 
located at the different sides to the node, 
 and   defined up to a shift by the same number, see Fig. \ref{double14}: 
$$
(c_{-, i}, c_{+, i}) \sim (c_{-, i}+a, c_{+, i}+a).
$$

\end{itemize}

\end{definition}

\begin{figure}[ht]
\centerline{\epsfbox{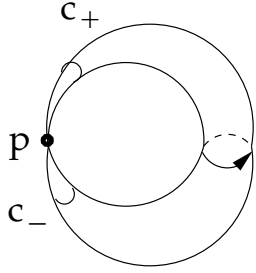}}
\caption{A codimension one stratum for a genus three surface with a two-loop lamination.} 
\label{double14}
\end{figure}

Cut the surface $\Sigma$ with a hyperbolic metric along the geodesic isotopic to a loop 
$\gamma_i$. We get a surface $\Sigma_i$ with geodesic boundary. The points of the Teichm\"uller 
space for $\Sigma$ are obtained from the ones for $\Sigma_i$ via a gluing 
procedure introducing one real parameter --  {\it the Dehn twist} 
along the loop $\gamma_i$. 

Here is a standard definition of the Dehn twist parameters. 
Take a universal cover $\widetilde {\Sigma_i}$ 
of $\Sigma_i$. 
It is obtained by cutting out from the hyperbolic plane ${\cal H}$ geodesic half discs 
bounded by the preimages of the boundary geodesic loops $\gamma_{\pm, i}$ on $\Sigma_i$. 
Choose a pair of boundary geodesics 
$g_{\pm}$ on $\widetilde {\Sigma_i}$ projecting to  $\gamma_{\pm, i}$. 
The geodesics $g_{\pm}$ are oriented, so that their orientations agree with the 
orientations of the boundary components $\gamma_{\pm, i}$ induced by the surface orientation.  
The Dehn twist parameter assigned to $\gamma_i$ parametrises orientation {\it reversing} 
isometries  $f: g_+\to g_-$:
\be \la{2.1.2012.1}
\{\mbox{Dehn twists for $\gamma_i$}\} \stackrel{\sim}{\lra} 
\{\mbox{maps $f: g_+ \to g_-$ such that $f(x+c) = f(x)-c$}\}. 
\ee

It is convenient for us to modify slightly this definition. Observe that a choice 
of the orientation of the loop $\gamma_i$ which enters in the definition of the 
stratum ${\cal X}^+_{S;\gamma; p_1, \ldots , p_k}$ provides a simultaneous 
choice of ends of the geodesics $g_\pm$. It provides therefore 
orientations of these geodesics, directed out of the chosen ends. The group $\R$ 
acts by translations of the geodesics so that a shift by a positive number moves a point according to the orientation. 
The Dehn twists assigned to $\gamma_i$ are parametrised by orientation {\it preserving}
isometries  $f: g_+\to g_-$: 
\be \la{2.1.2012.1a}
\{\mbox{Dehn twists for $\gamma_i$}\} \stackrel{\sim}{\lra} 
\{\mbox{maps $f: g_+ \to g_-$ such that $f(x+c) = f(x)+c$}\}. 
\ee

Let us glue the strata into a space ${\cal X}^+_{\Sigma; \gamma}$ 
so that pinching  
$\gamma_i$ to the node $p_i$ we get in the limit the corresponding stratum, such that the following condition holds:

\begin{itemize}

\item {\it The gluing transforms the Dehn twist action of $\R$ into the 
action of $\R$ provided by shifting the horocycle $c_i$ by $l \in \R$}.
\end{itemize}

The data (\ref{2.1.2012.1a}) is the same as a choice of a pair of horocycles 
centered at the chosen ends of the geodesics 
$g_+$ and $g_-$, defined up to their shifts by the same number, see Fig \ref{double12}. 
Indeed, given two such horocycles $c_+$ and $c_-$ 
there is a unique orientation preserving 
isometry $g_+ \to g_-$ identifying $c_+ \cap g_+$ and $c_- \cap g_-$.  
\begin{figure}[ht]
\centerline{\epsfbox{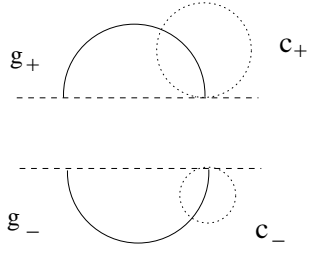}}
\caption{A pair of geodesics with a pair of horocycles centered at the chosen ends.}
\label{double12}
\end{figure}
Pinching the geodesic isotopic to the loop $\gamma_i$ on $\Sigma$ we  shrink the geodesics 
$g_{+}$ and $g_-$ to cusps, but keep  
a pair of horocycles centered at the cusps, defined up to a common shift. 
Thus in the limit we get a point of the 
 corresponding stratum. 

\vskip 3mm
The group $\R$ acts on (\ref{2.1.2012.1a}) by $(t_af)(x):= f(x+a)$, $a\in \R$.  
There is an action of 
the group $\R^{k}$ 
 on the space ${\cal X}^+_{\Sigma;\gamma}$: 
an element $a \in \R$ in the factor 
assigned to a loop $\gamma_i$ acts shifting by $a$ the Dehn parameter if $\gamma_i$ 
was not shrank to a node, and by shifting the horocycle $c_{+,i}$ by $a$ otherwise.  
It makes ${\cal X}^+_{\Sigma;\gamma}$ into a principal $\R^k$-fibration. 

The stratum assigned to
 $  \Sigma_{p_1, \ldots , p_k}$ 
 is fibered over the stratum of  the Weil-Peterson completion of the 
classical 
Teichm\"uller space of $\Sigma$ 
assigned to $\{\gamma_1, \ldots ,\gamma_k\}$. 
The latter stratum is of  
real codimension $2k$, while the former is of real codimension $k$. 
The Weil-Peterson stratum is the quotient of our stratum by the action of $\R^k$. 
Our strata lie inside of the space ${\cal X}^+_{\Sigma; \gamma}$, while the 
Weil-Peterson strata lie on the boundary of the Teichm\"uller space. 

There is a natural action of the group $(\Z/2\Z)^n$ on the  space ${\cal X}^+_{\Sigma; \gamma}$, where $n$ is the number of 
connected components of the lamination $\gamma$. It preserves the stratification, and 
 acts on the stratum ${\cal X}^+_{\Sigma;\gamma; p_1, \ldots , p_k}$ via the quotient $(\Z/2\Z)^{n-k}$, 
by changing the orientations of the $n-k$ loops of $\gamma$ which were not shrank to the nodes. 
The quotient ${\cal X}^+_{\Sigma; \gamma}/(\Z/2\Z)^n$ is a manifold with corners of depth $\leq n$, obtained by 
completion of the classical Teichmuller space of $\Sigma$.

\subsection{The modified Teichm\"uller space for the double.} 
\la{Sec2.3}
Let $S$ be an oriented hyperbolic surface with $n> 0$ holes $h_i$. 
Denote by $S^o$ 
the same surface with the opposite orientation. 
The double $S_{\cal D}$ of $S$ is defined by gluing 
the surfaces $S$ and $S^o$ along the corresponding parts of the  boundaries, 
see Fig \ref{double11}. 
It is an oriented  surface without holes. 
It carries a  simple lamination $\gamma$  
obtained by gluing the boundaries of $S$ and $S^o$. 
Cutting the double $S_{\cal D}$ 
along $\gamma$ we recover $S$ and $S^o$. 

\begin{definition} \label{D2q} 
The moduli space ${\cal D}^+_{S}$ is the space ${\cal X}^+_{S_{\cal D}; \gamma}$ 
for the lamination $\gamma$.
\end{definition}

\paragraph{Coordinates  on ${\cal D}^+_{S}$.} 
Denote by $h_1, ..., h_n$ the holes on $S$, and by 
$\partial h_1, ..., \partial h_n$ the corresponding boundary components of $S$. 
Shrink the holes $h_1, ..., h_n$ on $S$ to punctures $p_1, ..., p_n$, 
 getting a surface $S'$. An  {\it ideal triangulation} of $S'$
 is a triangulation of $S'$ 
with vertices at the punctures.

\begin{theorem} \label{COORD} Given an  {\it ideal triangulation} of $S'$,
 the  space ${\cal D}^+_{S}$ 
has a coordinate system which identifies it with $\R^{-6\chi(S)}$. 
\end{theorem}

\begin{figure}[ht]
\centerline{\epsfbox{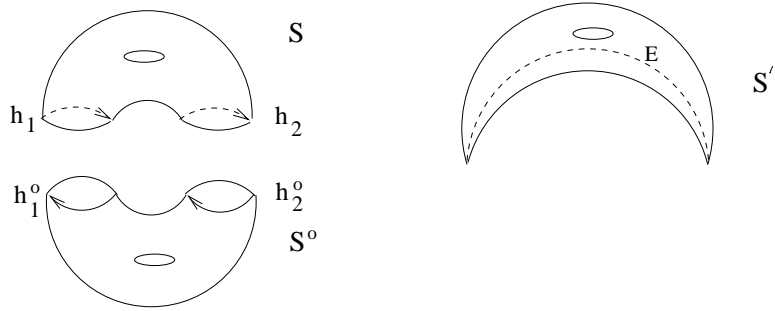}}
\caption{Left: gluing surfaces $S$ and $S^o$; Right: shrinking holes on $S$, getting a surface $S'$.}
\label{double11}
\end{figure}

{\bf Proof}. Take the universal cover 
$\widetilde S$ of $S$, and the universal cover 
$\widetilde {S^o}$ of $S^o$. Take an ideal edge $E$ connecting punctures 
$p_1$ and $p_2$ on $S'$, see Fig. \ref{double11}. 
Let $E^o$ be its mirror image on $S^o$.

Choose a pair of geodesics 
$g_1, g_2$ on $\widetilde S$ projecting to the boundary geodesics 
corresponding to the holes $h_1, h_2$.  The orientation of the loop $\gamma_i$ 
determines an end $e_i$ of the geodesic $g_i$. There is a geodesic $g_E$ on $S$ realizing 
the edge $E$ which spirals around the holes $h_1, h_2$ 
towards the ends $e_1, e_2$. Consider the geodesic $\widetilde E$ on $\widetilde S$ 
projecting to the geodesic $g_E$. It has the ends at $e_1, e_2$. 
\begin{figure}[ht]
\centerline{\epsfbox{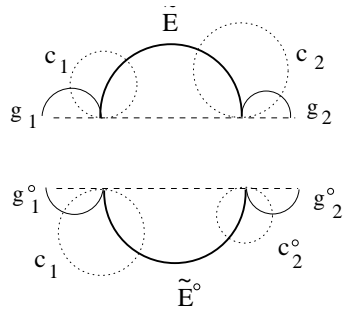}}
\caption{A geometric description of the coordinate $B_E$.}
\label{double13}
\end{figure}

A choice of the geodesics $g_1$ and $E$ determines uniquely the geodesic $\widetilde E$.  
Then the geodesic $g_2$ is determined uniquely by $\widetilde E$. 
Take a similar data $g^o_1, g^o_2$ and $\widetilde E^o$
on $\widetilde {S^o}$ assigned to  $E^o$. 
The gluing data contains a pair of 
horocycles $(c_1, c_1^o)$ centered at the ends $e_1, e_1^o$ of the geodesics 
$g_1, g_1^o$, defined up to a common shift, 
and a similar pair of horocycles $(c_2, c_2^o)$.

\begin{definition} The coordinate $b_E$ assigned to the edge $E$ is the difference of lengths of geodesics 
$\widetilde E$ and  $\widetilde E^o$, measured 
using pairs of horocycles $(c_1, c_1^o)$ and $(c_2, c_2^o)$,   
see Fig. \ref{double13}:
$$
b_E := l_{\widetilde E}(c_1, c_2) - l_{\widetilde E^o}(c_1^o, c_2^o) .
$$
\end{definition} 
Here $l_{\widetilde E}(c_1, c_2)$ 
is the distance between  
$\widetilde E\cap c_1$ and $\widetilde E\cap c_2$. Clearly 
$b_E$ does not depend on shift of pairs of horocycles 
$(c_1, c_1^o)$ and $(c_2, c_2^o)$. It is a new coordinate assigned to the edge $E$. 

The edge $E$ determines an ideal quadrilateral $(E_{12}, E_{23}, E_{34}, E_{41})$ with the diagonal $E = E_{13}$,  
see Fig. \ref{double5}. 
\begin{definition} The coordinate $x_E$ assigned to the edge $E$ is given by 
$$
x_E:= l_{E_{12}}(c_1, c_2) - l_{E_{23}}(c_2, c_3) + l_{E_{34}}(c_3, c_4) - l_{E_{41}}(c_4, c_1).  
$$
\end{definition} 
The coordinate $x_E$ is independent of the choice of horocycles $c_i$. 
It is the standard shear coordinate assigned to the edge $E$.

\subsection{The moduli space ${\cal D}^*_{S}$} 
\la{Sec3}

Let $S$ be a surface with holes. 
We introduce a moduli space ${\cal D}_{S}$, and identify the set of its 
positive points with the space ${\cal D}^+_{S}$. 

We start with a moduli space ${\cal D}^*_{S}$ which is an 
open part of the moduli space ${\cal D}_{S}$. 
Its advantage is that it can be defined as a moduli 
space of local systems on $S$.

Recall the canonical 
involution  $\sigma: {S}_{\cal D} \to {S}_{\cal D}$  
and the subgroup of the $\sigma$-invariant maps
\be \la{SUBGR}
\Delta_{SL_2} \subset {\rm Hom}(H_1({S}_{\cal D}, \Z), \Z/2\Z). 
\ee
Realising $\Z/2\Z$ as the center of the group $SL_2$, we make the group 
${\rm Hom}(H_1({S}_{\cal D}, \Z), \Z/2\Z)$ act 
on the space of twisted $SL_2$-local systems on $S_{\cal D}$.

\begin{definition} \label{D2} The moduli space ${\cal D}^*_{S}$ parametrises 
the orbits of the subgroup  $\Delta_{SL_2}$ 
 on the moduli space of twisted framed $SL_2$-local systems on ${S}_{\cal D}$.
\end{definition}

The moduli space 
${\cal X}_{S}:= {\cal X}_{PGL_2,S}$ parametrises $PGL_2$-local systems on $S$ with a framing, i.e. 
a choice of an 
eigenspace for the monodromy around every boundary component of $S$.

 Let us define an atlas on the  space ${\cal D}^*_{S}$ 
whose coordinate 
systems are parametrized by {ideal triangulations} $T$ of $S'$. 
Given such a $T$, we define  
a rational coordinate system $(B_E, X_E)$, where $E$ runs through the edges of 
$T$. This atlas  on ${\cal D}^*_{S}$ has a structure of the cluster symplectic  
double of the Poisson moduli space 
${\cal X}_{S}$.

Given a pair of vectors $v_1, v_2$ in a two-dimensional symplectic vector space, 
denote by $\Delta(v_1, v_2)$ the area of the 
parallelogram spanned by these vectors. 
Take an ideal  quadrilateral with vertices parametrised by a set 
$\{1,2,3,4\}$. Pick a non-zero vector in the fiber of the two dimensional 
vector bundle associated with ${\cal L}$ over each of the vertices of the quadrilateral which projects to the eigenline defining the framing at that vertex. 
We get a pair of configurations 
of four non-zero vectors in a two dimensional vector space: 
$$
(l_1, l_2, l_3, l_4) \quad \mbox{and} \quad (l^o_1, l^o_2, l^o_3, l^o_4), 
$$ 
well defined up to an action of the group $(\C^*)^4$, where 
an element $\lambda_i$ from the $i$-th factor $\C^*$ multiplies 
each of the vectors $l_i$ and $l_i^o$ by $\lambda_i$. 
The vectors $l_i$ and $l_i^o$ are assigned to 
the vertex $i$ of the quadrilateral. 
The $B$- and $X$-coordinates assigned to the edge $(1,3)$ are 

$$
B_{13}:= \frac{\Delta(l^o_1, l^o_3)}{\Delta(l_1, l_3)}, \qquad 
X_{13}:= \frac{\Delta(l_1, l_4)\Delta(l_2, l_3)}
{\Delta(l_1, l_2)\Delta(l_3, l_4)}.
$$
Multiplying both $l_i$ and $l^o_i$ by $\lambda_i$ we do not 
 change  $B_{13}$. Clearly one has

\begin{equation} \label{6.20.07.1}
X^o_{13}= \frac{\Delta(l^o_1, l^o_2)
\Delta(l^o_3, l^o_4)}{\Delta(l^o_1, l^o_4)\Delta(l^o_2, l^o_3)} 
= X_{13}^{-1}\frac{B_{12}B_{34}}{B_{14}B_{23}}.
\end{equation}

\begin{theorem} \label{MT100} (i) 
The rational functions $(B_E, X_E)$ 
assigned to an ideal triangulation $T$ of $S'$ provide a  
coordinate system on 
the moduli space ${\cal D}^*_{S}$. 

(ii) The atlas given by these coordinate systems is 
the cluster atlas 
for the double of 
${\cal X}_{S}$.

(iii) There is a canonical isomorphism ${\cal D}^*_{S}(\R_{>0}) = {\cal D}^+_{S}$.
\end{theorem}

\begin{proof} (i) The functions $B_E$ do not change  under the action of the 
subgroup (\ref{SUBGR}) on the moduli space of framed twisted $SL_2$-local systems on 
${S}_{\cal D}$. Indeed, acting by this subgroup 
 we alter $\Delta(l^o_1, l^o_3)$ and  $\Delta(l_1, l_3)$ by the same 
sign.  So the functions $B_E$ live on the space of orbits of the subgroup (\ref{SUBGR}). 
The same is evidently true for the functions $X_E$. 

Observe that if we did not take the quotient by the action of the subgroup (\ref{SUBGR}), 
the functions $(B_E, X_E)$ would not have the chance to be coordinates. 
 
The claim that they are coordinates 
has the same proof as the proof of the general Theorem \ref{MTD}, so we skip it here.  

(ii) We claim that our $(B,X)$-coordinates behave under a flip just like the ones 
on the symplectic double  
${\cal X}_{S}$. The $X$-coordinates are the same as for  
${\cal X}_{S}$. 
Let us calculate how the coordinate $B_{13}$ changes 
under the flip at the edge $(1,3)$, see Fig. \ref{double5}. 
Recall the Pl\"ucker relation
$$
\Delta(l_1, l_3)\Delta(l_2, l_4) = \Delta(l_1, l_2)\Delta(l_3, l_4) + 
\Delta(l_1, l_4)\Delta(l_2, l_3),  
$$
and a similar relation for the configuration 
$(l^o_1, l^o_2, l^o_3, l^o_4)$. The flipped $B$-coordinate 
$B_{24}$ equals
$$
B_{24}:= \frac{\Delta(l^o_2, l^o_4)}{\Delta(l_2, l_4)} = 
\frac{\Bigl(\Delta(l^o_1, l^o_2)\Delta(l^o_3, l^o_4) + 
\Delta(l^o_1, l^o_4)\Delta(l^o_2, l^o_3)\Bigr)\Delta(l_1, l_3)}
{\Bigl(\Delta(l_1, l_2)\Delta(l_3, l_4) + 
\Delta(l_1, l_4)\Delta(l_2, l_3)\Bigr)\Delta(l^o_1, l^o_3)} = 
\frac{\Bigl(B_{12}B_{34} + X_{13} B_{14}B_{23}\Bigr)}
{\Bigl(1 + X_{13} \Bigr)B_{13}}.
$$
 This agrees with the mutation formula (\ref{4.28.03.11x}) for the $B$-coordinates 
from Section \ref{Sec6}.

\vskip 3mm
(iii) The space  ${\cal X}^+_{S}:= {\cal X}_{S}(\R_{>0})$ 
is identified with the modified Teichm\"uller spaces parametrising 
complex structures on $S$ plus eigenvalues of the 
monodromies around the boundary components (\cite{FG1}). The canonical projection
\begin{equation} \label{6.27.07.1}
\pi: {\cal D}^*_{S_{\cal D}}(\R_{>0}) \lra {\cal X}^+_{S} \times {\cal X}^+_{S^o} 
\end{equation} 
is a principal fibration with the fiber $\R^k$. Its image is a linear 
subspace in the logarithmic coordinates given by the condition that the monodromies 
around the holes $h_i$ and $h_i^o$ coincide. 
On the other hand, cutting the double $S_{\cal D}$ along $\gamma$ we get a projection 
\begin{equation} \label{6.27.07.2}
\pi: {\cal D}^+_{S_{\cal D}} \lra {\cal X}^+_{S} \times {\cal X}^+_{S^o}  
\end{equation} 
with the same image, which is also a principal $\R^k$-fibration. 
So to construct an isomorphism ${\cal D}^+_{S_{\cal D}} \to {\cal D}^*_{S_{\cal D}}(\R_{>0})$ 
it is sufficient to define a map of principal $\R^k$-bundles 
(\ref{6.27.07.2}) $\to$ (\ref{6.27.07.1}) over the same base. 

 The open stratum ${\cal X}^+_{S_{\cal D}; \gamma, \emptyset}$
parametrizes pairs (a complex structures on $S_{\cal D}$, 
  a choice  of an orientation for each loop $\gamma_i$). 
 Translating into  the language of positive local systems (\cite{FG1}, Section 11), 
${\cal X}^+_{S_{\cal D}; \gamma, \emptyset}$ parametrises pairs $({\cal L}, \beta)$,
  where ${\cal L}$ is 
 a positive $PGL_2(\R)$-local system on $S_{\cal D}$ (i.e. $X_E>0$ for all coordinates $X_E$ of 
a coordinate system on ${\cal X}^+_{S}$), and 
 $\beta$ encodes choice of an eigenspace of the  
monodromy of ${\cal L}$ for each loop $\gamma_i$. 
Let us define an open $\R^k$-equivariant embedding  
$$
j: {\cal X}^+_{S_{\cal D}; \gamma, \emptyset} \hra
 {\cal D}^*_{S}(\R_{>0}).
$$ 
Cutting  $S_{\cal D}$ along  $\gamma$ and restricting the pair $({\cal L}, \beta)$ 
to the obtained surface 
we get framed $PGL_2(\R)$-local systems on $S$ and $S^o$.  Since they 
arose from points of the Teichm\"uller space, they are positive.  
Their monodromies around the loops $\partial h_i$ and $\partial h^o_i$ coincide, 
and conjugate to a diagonal matrix different from the identity, with positive  diagonal entries. 
Next, the group $\R^*_+ \stackrel{\sim}{=} \R$ acts on the 
gluing data $\alpha$ in Definition \ref{D2} restricted to $\partial h_i$ by 
multiplying it by $\lambda_i \in \R^*_+$, as well as 
on the Dehn twist parameters for $\partial h_i$. 
There is an $\R$-equivariant bijection 
$$
\mbox{$\{$Dehn twist parameters for 
$\partial h_i$, an orientation of 
$\partial h_i$$\}$} \stackrel{\sim}{\lra} \mbox{$\{$Gluing data for $\partial h_i$ in Def. \ref{D2}$\}$}
$$
Moreover, by the very definition, $B_E = {\rm exp}(b_E)$, $X_E = {\rm exp}(x_E)$. Thus $B_E>0$. 
We get the embedding $j$. It extends to an $\R^k$-equivariant embedding of 
${\cal D}_S^+ \hra {\cal D}^*_{S}(\R_{>0})$. Since both spaces are principal 
$\R^k$-fibrations over the same base, we are done. \end{proof}

\section{Special coordinates on the symplectic double for general $\G$} \la{Sec4}

\subsection{Main construction}

Let us 
construct a positive  atlas on the moduli space ${\cal D}^*_{\G, \bS}$, 
whose coordinate 
systems are parametrized by the same 
set as the ones on ${\cal X}_{\G,
\bS}$, and have the properties of the cluster symplectic double atlas. 
Choose an ideal  triangulation $T$ of $\bS'$.

\paragraph{The $X$-coordinates on ${\cal D}^*_{\G, \bS}$.} 
Given a triangulation $T$, 
they are the inverse images 
of the $X$-coordinates on ${\cal X}_{\G, \bS}$ for the projection 
${\cal D}^*_{\G,\bS} \to {\cal X}_{\G,\bS}$ given by the restriction from 
$\bS_{\cal D}$ to $\bS$. 

\paragraph{The $B$-coordinates on ${\cal D}^*_{\G,\bS}$.} 
Choose a triangle $t$ of the triangulation $T$. 
Our goal is to produce a pair of points 
in ${\rm Conf}_3({\cal A}_{\widetilde \G})$ assigned to the triangle $t$ 
on $\bS'$ and a point of ${\cal D}^*_{\G,\bS}$, well defined up to a diagonal 
action of the group $H^3_{\widetilde \G}$, that is a point of 
 \be \la{7.26.14.1}
\Bigl({\rm Conf}_3({\cal A}_{\widetilde \G})
\times {\rm Conf}_3({\cal A}_{\widetilde \G})\Bigr)/H_{\widetilde \G}^3.
\ee

Denote by $p_1, p_2, p_3$ the vertices of the triangle $t$. They are 
either punctures or marked points on $\bS'$. 
Pick a path $E$ connecting  $p_i$ and $p_j$ on $\bS'$.

A) Let us consider first the case when both $p_i$ and $p_j$ are punctures. 
Denote by $h_i$ and $h_j$ the corresponding holes on $\bS$. 

Choose a point $x_s$ on the boundary loop $\gamma_s$ of the hole $h_s$. 
Take a path 
$E^+ \subset \bS \subset \bS_{\cal D}$ connecting points $x_i$ and $x_j$, which 
shrinks to a path isotopic to $E$ as we shrink the holes $h_i$ and $h_j$ to the 
punctures $p_i$ and $p_j$.  
The isotopy class of $E^+$ considered up to winding around the loops $\gamma_i$ 
and $\gamma_j$ is uniquely defined. 
Let $E^-$ be the mirror of $E^+$  under the involution of $\bS_{\cal D}$ 
interchanging $\bS$ and $\bS^o$. 

Pick a triangle $t^+ \subset \bS \subset \bS_{\cal D}$ with vertices at $x_i$'s which shrinks to $t$. 
Let $t^-\subset \bS^o \subset \bS_{\cal D}$ be its mirror. 
Points of ${\cal D}^*_{\G,\bS}$ are orbits of the group $(\ref{subgrd})$  
acting on  the following data:

\begin{enumerate}

\item A twisted $\widetilde \G$-local system 
${\cal L}$ on $\bS_{\cal D}$ and 

\item 
A flat section $\beta_{\gamma_i}$ of the flag bundle 
${\cal L}_{\cal B}$ over each loop $\gamma_i$.
\end{enumerate}

Pick an decorated flag $A_i$ at the fiber of ${\cal L}_{\cal A}$ over the point 
$x_i$ projecting to 
 the restriction of $\beta_{\gamma_i}$ to $x_i$. 
Since the triangle $t^+$ is contractible, the decorated flags at its vertices 
provide a configuration of three decorated flags. 
The same for  $t^-$. We get two triples:
\begin{equation} \label{11.5.06.2}
(a_1, a_2, a_3) \in {\rm Conf}_3({\cal A}_{\widetilde \G}) 
\quad \mbox{and} \quad 
(a^o_1, a^o_2, a^o_3) \in 
{\rm Conf}_3({\cal A}_{\widetilde \G}).
\end{equation}

\begin{lemma} \label{11.5.06.5q}
The triples (\ref{11.5.06.2})  are well defined up to the diagonal action of the group 
$H_{\widetilde \G}^3$, producing a point in (\ref{7.26.14.1}). 
\end{lemma}

\begin{proof}  Follows immediately from the two observations: 

(i) The decorated flag $A_i$ is well 
defined up to the action of the group $H_{\widetilde \G} $. 

(ii) Altering the triangle $t^+$ 
by rotating the point $x_i$  around the loop $\gamma_i$, we 
alter the decorated flag $A_i$ by the monodromy around the loop. 
So 
$a_i$ and $a^o_i$ are multiplied by the same 
element of $H_{\widetilde \G} $. 
\end{proof}

B) The case when one or two of the endpoints $p_i$, $p_j$ 
of $E$ are marked points is treated similarly, 
and is in fact simpler. 

\vskip 3mm
According to Section 8 of \cite{FG1}, 
there is a set $I_t$ parametrising 
the $A$-coordinates on the configuration space 
${\rm Conf}_3({\cal A}_{  \widetilde \G})$. 
So each  $i\in I_t$ provides two numbers: 
$A_i(a_1, a _2, a _3)$ and $A_i(a^o_1, 
a^o_2, a^o_3)$. The coordinate $B_{t,i}$ related 
to the triangle $t $ is defined as 
their ratio: 
\begin{equation} \label{11.5.06.4}
B_{t,i}:= \frac{A_i(a^o_1, 
a^o_2, a^o_3)}{A_i(a _1, a _2, a _3)}, \quad i\in I_t.
\end{equation}

\begin{lemma} \label{11.5.06.5}
Ratio (\ref{11.5.06.4}) does not depend on the choices 
in the construction of triples (\ref{11.5.06.2}). 
\end{lemma}

{\bf Proof}. 
The only fact we need is the following property 
of the $A$-coordinates on  ${\rm Conf}_3({\cal A}_{ \widetilde  \G})$:

\begin{lemma} \label{11.5.06.7}
Each $i \in I_t$ 
determines a character $\chi_i$ of the group $H_{\widetilde \G}^3$ such that 
one has
$$
A_i(h_1a_1, h_2a_2, h_3a_3) = \chi_i(h_1, h_2, h_3)A_i(a_1, a_2, a_3), 
\quad  i\in I_t, ~~~ \forall (h_1, h_2, h_3) \in H_{\widetilde \G}^3.
$$
\end{lemma}

{\bf Proof}. Follows from the definition of positive 
atlas on ${\rm Conf}_3({\cal A}_{\widetilde \G})$ 
in Section 8 of \cite{FG1}.  

\vskip 3mm
Lemma \ref{11.5.06.5}  follows immediately from Lemmas \ref{11.5.06.7} and 
\ref{11.5.06.5q}. 
\vskip 3mm

Denote by ${\rm D}_{\G, \bS}$ the cluster symplectic double 
of the cluster variety 
${\cal X}_{\G,\bS}$. We distinguish it from the moduli space 
${\cal D}_{\G, \bS}$. It is easy to check that
\be \la{dimeq}
{\rm dim}{\rm D}_{\G, \bS} = {\rm dim}{\cal D}^*_{\G, \bS} = 2  {\rm dim}{\cal X}_{\G, \bS}.
\ee

\begin{theorem} \label{MTD} There is a canonical rational surjective at the generic point 
map of spaces
\be \la{JUly20.14.1}
{\cal D}^*_{\G,\bS} \lra{\rm D}_{\G, \bS}.
\ee
\end{theorem}

The proof of Theorem \ref{MTD} will show that the map 
(\ref{JUly20.14.1}) is a finite cover 
at the generic point. 

\bc \la{5.20.15.1}
The map (\ref{JUly20.14.1}) is a birational isomorphism. 
\ec

A proof of Conjecture \ref{5.20.15.1} was claimed by D. Allegretti \cite{A1}. 

\begin{proof} The claim of Theorem \ref{MTD} is equivalent to the following: 

\begin{enumerate}

\item The rational functions 
$(X^T_i, B^T_i)$ on the 
moduli space ${\cal D}^*_{\G,\bS}$ assigned to an ideal triangulation $T$ of $\bS'$ 
are independent;

\item
The functions  $(X^T_i, B^T_i)$ for different ideal triangulations $T$ 
are related by cluster transformations for the symplectic 
double 
of the cluster ${\cal X}$-variety ${\cal X}_{\G,\bS}$. 
\end{enumerate}

We start from the proof of the Claim 2.

\paragraph{Proof of Claim 2.} Let us show that the  $(B,X)$-coordinates  on ${\cal D}^*_{\G,\bS}$ 
for different ideal triangulations $T$ are related by cluster double transformations.  

A flip $T \to T'$ at an edge $E$ of $T$ 
is decomposed into a composition of mutations as in Section 10 of \cite{FG1}. 
We need to show that this sequence of mutations 
transforms the $(B,X)$-coordinates assigned to $T$ 
to the ones for $T'$. 

The double $X$-coordinates are just the usual cluster $X$-coordinates 
on ${\cal X}_{\G, \bS}$, so the claim follows from the 
corresponding claim for the $X$-coordinates proved in {\it loc. cit}. 

Our $B$-coordinates on ${\cal D}^*_{\G,\bS}$ 
are defined as ratios 
of appropriate $A$-coordinates. Precisely, 
consider an ideal quadrilateral with vertices 
parametrised by the set $\{1,2,3,4\}$, so that the 
$E$ is the diagonal $(1,3)$, see Fig. \ref{double5}. We assigned to 
it a pair of configurations of decorated flags 
$$
(A_1, A_2, A_3, A_4) \quad \mbox{and} \quad (A^o_1, A^o_2, A^o_3, A^o_4). 
$$ 
defined up to the simultaneous action of the group $H_{\widetilde \G}^4$, given by 
$$
(A_1, A_2, A_3, A_4)\lms  (h_1A_1, h_1A_2, h_1A_3, h_1A_4), \qquad 
 (A^o_1, A^o_2, A^o_3, A^o_4)\lms  (h_1A^o_1, h_1A^o_2, h_1A^o_3, h_1A^o_4).
$$
The $A$-coordinates we use are cluster $A$-coordinates ({\it loc. cit.}).

To prove that our 
$B$-coordinates on ${\cal D}^*_{\G,\bS}$ 
behave under the mutations just as the cluster $B$-coordinates, 
consider the diagram
$$
{\rm Conf}_4({\cal A}) \times {\rm Conf}_4({\cal A}) 
\stackrel{\varphi}{\lra} 
\Bigl({\rm Conf}_4({\cal A}) \times {\rm Conf}_4({\cal A})\Bigr)/H_{\widetilde \G}^4 \stackrel{\pi}{\lra}
{\rm Conf}_4({\cal B}) \times {\rm Conf}_4({\cal B}). 
$$
The $B$-coordinates live on the middle space (by the same argument as in 
Lemma \ref{11.5.06.5q}).  
Since $\varphi$ is surjective, to  check that 
they transform as the cluster  $B$-coordinates it is sufficient to do it 
for the lifted coordinates $\varphi^*B_i$. 
To prove the latter we employ the fact  that the map $\varphi$ 
commutes with mutations: the computation checking this was carried out in the 
proof of the part ii) of Theorem  \ref{MT100}. 

\paragraph{Proof of Claim 1.}

Given a triangulation $T$, the rational functions $(X^T_i, B^T_i)$ 
provide a rational map $\psi_{\bf q}: {\cal D}_{\G,\bS} \lra {\cal D}_{\bf q}$, 
where ${\bf q}$ is the seed assigned to the triangulation $T$, and 
${\cal D}_{\bf q}$ is the corresponding seed torus. 
There is  a diagram
\begin{equation} \label{11.15.06.1}
\begin{array}{ccc}
{\cal D}_{\G,\bS}&\stackrel{\pi}{\lra}
&\Delta_{{\cal X}_{\G,\bS}} \subset{\cal X}_{\G,\bS}\times {\cal X}_{\G,\bS^o}\\
&&\\
\psi_{\bf q}\downarrow &&\downarrow \sim \\
&&\\
{\cal D}_{\bf q} &\stackrel{{\pi_{\bf q}}}{\lra}&
\Delta_{{\cal X}_{\bf q}} \subset {\cal X}_{\bf q}\times {\cal X}_{{\bf q}^{\rm o}}\\
\end{array}
\end{equation} 

Here $\Delta_{{\cal X}_{\G,\bS}} $ and $\Delta_{{\cal X}_{\bf q}}$ are the diagonals in the corresponding products, defined as the invariants 
of the involution $\sigma$. In particular, $\Delta_{{\cal X}_{\G,\bS}} = 
{\cal X}^{\rm red}_{\G,\bS_{\cal D} - \gamma} $. 

The map $\pi$ is the restriction map. 

The right vertical 
map is a birational isomorphism given by the rational map to the cluster seed torus for  
${\cal X}_{\G,\bS}\times {\cal X}_{\G,\bS^o}$. 

The square is commutative by construction: this is evident for the projection to 
${\cal X}_{\G,\bS}$, and follows from the definition of the $B$-coordinates 
 for the projection to 
${\cal X}_{\G,\bS^o}$. Indeed, the functions $X^o_i$ on ${\cal D}_{\G,\bS}$ defined by  
cluster formulas (\ref{zx1q1}) 
coincide with the  $X_i$-functions for the mirror triangulation $T^o$ on 
$\bS^o \subset \bS_{\cal D}$.

By Theorem \ref{11.11.06.2} the map $\pi$ 
at the generic point is a principal fibration with the structure group 
$H_{\widetilde \G}^k$, where $k$ is the number of holes without marked points on $\bS$.

The map $\psi_{\bf q}$  at the generic point is a map of fibrations. 
The map $\psi_{\bf q}$ transforms faithfully 
the action of the torus $H_{\widetilde \G}^k$. Thus, thanks to (\ref{dimeq}), the map $\psi_{\bf q}$ 
is surjective at the generic point.

\end{proof} 

\subsection{Functions $B(\alpha)$ on the 
moduli space ${\cal X}_{\G, \bS; \gamma}$} \la{SecBcoord}

We start with  a generalisation of the moduli space introduced 
in Section \ref{SecX}. 
Let $\bS$ be a decorated surface, and $\gamma'$ 
a simple lamination on it, given by a collection 
of simple non-intersecting loops. 
The {\it punctured boundary} of $\bS$ is the boundary of $\bS$ minus the 
marked points:
$$
\partial^*\bS := \partial\bS - \{\mbox{\rm marked points}\}.
$$
Let us set
$$
\gamma:= \gamma' \cup \partial^*\bS.
$$

\bd
The moduli space ${\cal X}_{\G, \bS; \gamma}$ parametrises 
pairs $({\cal L}, \beta)$ where ${\cal L}$ is a $\G$-local system on 
$\bS$, and $\beta$ is a framing, given by a flat section of the associated flag local system ${\cal L}_{\cal B}$ on $\gamma$. 
\ed

Take an ordered collection of points $z_1, ..., z_k$ 
on  $\gamma$. For each consequitive 
pair of points $z_iz_{i+1}$ consider an arbitrary path $\alpha_{i, i+1}$ 
on $\bS$ which does not intersect $\gamma$ and connects the points 
$z_i$ and $z_{i+1}$ . It is oriented from $z_i$ towards $z_{i+1}$. 
Travelling along  these paths we get a loop 
$$
\alpha(z_1, ..., z_k) = \alpha_{1,2} \circ 
\alpha_{2, 3} \circ \ldots \circ \alpha_{k, 1}. 
$$
The loop can have selfintersections. We consider it up to isotopies 
such that: 
$$
\mbox{the paths $\alpha_{i, i+1}$ end on $\gamma$, and their interier parts 
do not intersect $\gamma$.}  
$$

 Let $({\cal L}, \beta)$ be a framed $\G$-local system on $(\bS; \gamma)$. 
Just as in the definition of the $B$-coordinates, 
pick a decorated flag $A_{z_i}$ in the fiber of  the decorated flag local system  ${\cal L}_{\cal A}$ 
at the point $z_i$ which projects to the flag $B_{z_i}$ 
in the fiber of
${\cal L}_{\cal B}$ over $z_i$ provided by 
 the framing $\beta$. Transporting  
the decorated flags 
$A_{z_i}$ and $A_{z_{i+1}}$ along the arc $\alpha_{i, i+1}$ 
into the same point of the arc, 
we get a 
configuration  of two decorated flags, denoted by 
$(A_{z_i}, A_{z_{i+1}})_\alpha$. 

Recall the $H$-invariant, given by 
the birational isomorphism
$$
h: {\rm Conf}_2({\cal A}_\G) \stackrel{\sim}{\lra }H. 
$$
It has the following property:
\be \la{redx}
h(tA_1, A_2) = th(A_1, A_2), ~~~~ h(A_1, tA_2) = 
w_0(t)h(A_1, A_2), ~~t \in H.
\ee

We apply the $H$-invariant map to 
the configuration $(A_{z_i}, A_{z_{i+1}})_\alpha$, getting
$$
h_\alpha(A_{z_i}, A_{z_{i+1}}):= h((A_{z_i}, A_{z_{i+1}})_\alpha)\in H. 
$$
Consider an alternating product
\be \la{Bf}
B(\alpha):= \frac{h_\alpha(A_{z_1}, A_{z_{2}})h_\alpha(A_{z_3}, A_{z_{4}}) \ldots }
{w_0h_\alpha(A_{z_2}, A_{z_{3}}) w_0h_\alpha(A_{z_4}, A_{z_{5}}) \ldots w_0h_\alpha(A_{z_{k}}, A_{z_{1}})}.
\ee
Thanks to (\ref{redx}),  rescaling the flag $A_{z_i} \lms tA_{z_i}$ we do not change the $B(\alpha)$. So we get a rational function $B(\alpha)$ on the space 
${\cal X}_{\G, S; \gamma}$, which assigns to a framed $\G$-local system 
$({\cal L}, \beta)$ the value of the invariant  $B(\alpha)$. 

The cyclic shift of the points $s: (z_1, ..., z_k) \lms (z_2 , ..., z_k, z_1)$  
changes the $B(\alpha)$ as follows:
$$
B(s(\alpha))= w_0(B(\alpha)^{-1}).
$$

\paragraph{Examples.} 1. If our surface is the double $\bS_{\cal D}$, and 
$\alpha_E$ is a loop on $\bS_{\cal D}$ obtained by doubling 
an ideal edge $E$ on the original surface, then $B(\alpha)$ is just the 
$B$-coordinate function $B_E$. 

2. When $\G = SL_2$ and 
$\alpha$ is a $4$-gon on $\bS$ with vertices 
at the  $\gamma$, the function $B(\alpha)$ 
is a generalisation of the $X$-coordinate. To get the latter 
we restrict to a contractable quadrilateral. 
Then $B(\alpha)$ is the cross-ratio of the configuration of 
four points on $P^1$ 
provided by the framing at the vertices of the quadrilateral. 
\vskip 3mm

It would be interesting to calculate the function $B(\alpha)$ 
in a cluster double coordinate system related to an ideal
 triangulation of the half $\bS\subset \bS_{\cal D}$. 

For $\G=SL_2$ the functions $B(\alpha)$ is studied by Dylan Allegretti \cite{A}, who discovered their close relationship  
to the $F$-polynomials of Fomin-Zelevinsky \cite{FZIV}.

\section{A complex analog of Fenchel-Nielsen coordinates}

\subsection{Construction of coordinates for $G=SL_2$} 

Let $S$ be a closed surface, i.e.  a surface without boundary and punctures. 
Let $g$ be the genus of $S$. We assume that $g>1$. Consider 
a collection $3g-3$ simple non-intersecting loops   
on $S$ which determine a pair of pants decomposition 
of $S$: cutting $S$ along these loops we get $2g-2$ pair of pants. 
The gluing pattern is described by a trivalent graph $\Gamma$: its vertices $v$ 
correspond to pairs of pants denoted ${\cal P}_v$, and its edges $E$ correspond to  
the loops, denoted $\alpha_E$. So such a graph $\Gamma$ has 
Betti number $g$; it has $2g-2$ vertices and $3g-3$ edges. 
Denote by ${\cal V}_\Gamma$ and ${\cal E}_\Gamma$ the sets of the vertices and edges of the graph $\Gamma$. 

The homology classes $[\alpha_E]$ of the loops 
generate a Lagrangian sublattice ${\rm L}_\alpha$ of $H_1(S, \Z)$. Indeed, 
each pair of pants gives a relation, and there is a single relation between these relations.  
So its rank is $(3g-3)-(2g-2)+1=g$. 
Denote by $\Z[X]$ the free abelian group 
generated by a set $X$. We arrive at  an isomorphism of abelian groups
\be \la{Lagiso}
{\rm Coker}\Bigl(\Z[{\cal V}_\Gamma] \lra \Z[{\cal E}_\Gamma]\Bigr) = {\rm L}_\alpha \subset H_1(S, \Z).
\ee

Let us define a dual collection of loops  $\{\beta_E\}$. 
Let us choose once forever an orientation of the graph $\Gamma$. 
Let $v_E^+$ and $v_E^-$ be the vertices of the edge $E$, so that $E$ is oriented from  $v_E^+$ to $v_E^-$. 
One can have $v_E^+ = v_E^-$, in which case $E$ is a loop. 

The  pairs of pants 
${\cal P}_{v_E^+}$ and ${\cal P}_{v_E^-}$ contain $\alpha_E$. They coincide if $E$ is a loop. 
Denote by $\beta_E^+$ a half loop on ${\cal P}_{v_E^+}$ shown on the top left of Fig \ref{id2}. 
It intersects the loop $\alpha_E$ at two points.  Denote by $\beta_E^-$ a similar 
half loop on ${\cal P}_{v_E^-}$ intersecting $\alpha_E$ at the same two points. 
The orientation of the 
half loops does not play any role. Set $\beta_E:= \beta_E^+\cup \beta_E^-$. 
See Fig. \ref{id2} and \ref{fn1}.

\begin{figure}[ht]
\centerline{\epsfbox{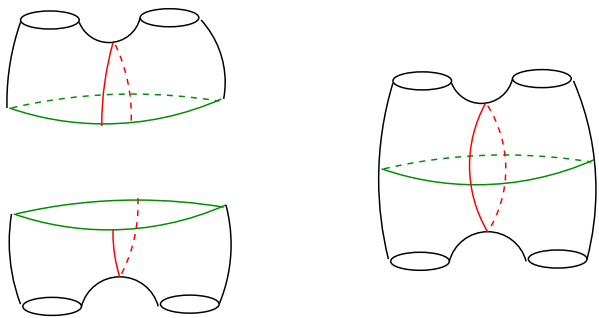}}
\caption{$E$ is not a loop. Gluing two pairs of pants along a green boundary loop we get four holed sphere. 
The complimentary red loop $\beta_E$ is obtained by gluing two red half loops. }
\label{id2}
\end{figure}

\begin{figure}[ht]
\centerline{\epsfbox{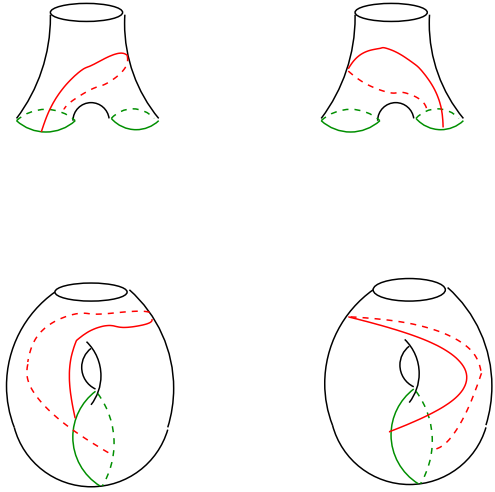}}
\caption{$E$ is  a loop. Gluing green boundary loops on a pair of pants we get a torus with a hole $T_E^o$. 
 On the left: a half loop on a pair of pants ending on the left boundary circle 
 is glued into a  half loop $\beta_E^+$ on $T_E^o$. 
On the right: a similar construction of a  half loop $\beta_E^-$. The loop $\beta_E$ on $T_E^o$  
is the union of 
$\beta_E^+ \cup \beta_E^-$.}
\label{fn1}
\end{figure}

So we get two collections of loops: 
$\{\alpha_E\}$ and $\{\beta_E\}$. 

Let $\alpha$ be a lamination on $S$ given by the union of the loops $\alpha_E$:
$$
\alpha = \cup_{E\in{\cal E}_\Gamma } \alpha_E.
$$
Then the coarse moduli space ${\rm Loc}_{SL_2, S; \alpha}$ from Section \ref{Sec5} 
parametrizes pairs $({\cal L}, \varphi)$ where ${\cal L}$ is a twisted 
$SL_2$-local systems on $S$, and $\varphi$ is a framing of ${\cal L}$ 
over the lamination $\alpha$, which amounts to a  choice of an eigenline of the monodromy 
of ${\cal L}$ along each of the loops $\alpha_E$. 

Forgetting the framing, we get  a $2^{3g-3}:1$ cover of the coarse moduli space ${\rm Loc}_{SL_2, S}$ of 
twisted $SL_2$-local systems on $S$:
$$
\pi_\alpha: {\rm Loc}_{SL_2, S; \alpha} \lra {\rm Loc}_{SL_2, S}.
$$

\paragraph{Complex analogs of Fenchel-Nielsen coordinates for  $SL_2$.}
Given two complimentary sets of loops   $\{\alpha_E, \beta_E\}$, 
let us define 
a collection of rational functions $\{M_E, B_E\}$ 
on the space ${\rm Loc}_{SL_2, S; \alpha}$, parametrized by the edges $E$ of the graph $\Gamma$.

A) Take a   loop $\alpha_E$. Our choice of an
 orientation of the edge $E$ provides an 
orientation of the loop $\alpha_E$ such that the pair of pants ${\cal P}_{v_E^+}$ is on the left. 
 
Then, given 
a twisted framed $SL_2$-local system $({\cal L}, \varphi)$ on $S$, the 
monodromy along the loop $\alpha_E$  preserves the one dimensional subspace determined by the framing. 
The eigenvalue  $\mu_E$ of the monodromy in this subspace provides a function 
$$
M_E: {\rm Loc}_{SL_2, S; \alpha} \lra \C^*, ~~~~ ({\cal L}, \varphi) \lms \mu_E.
$$

B) Take a  loop $\beta_E$. Let us define a rational function 
\be \la{RFBj}
B_E: {\rm Loc}_{SL_2, S; \alpha} \lra \C^*.
\ee
The loops $\beta_E$ and $\alpha_E$ intersect at two points $x,y$. So 
$\beta_E - \{x \cup y\}$  is a union of two arcs:
$$
\beta_E - \{x \cup y\} = \beta_E^+ \cup  \beta_E^-.
$$

Take a non-zero vector $v_x$ at the eigenline $L_x \subset {\cal L}_x$ at the point $x$ 
 of the monodromy of ${\cal L}$ 
along the loop $\alpha_E$. It is the eigenline providing a framing over $\alpha_E$. 
Take a similar vector $v_y \in L_y \subset {\cal L}_y$ 
over the point $y$. Moving the vectors $v_x, v_y$ along the arc $\beta_E^+$ 
to the same point, we define a number
\be \la{4.8.14.1}
\Delta_{\beta_E^+}(v_x, v_y)\in \C^*.
\ee
We use the fact that ${\cal L}$ is a \underline{twisted} $SL_2$-local system on $S$: otherwise 
number (\ref{4.8.14.1}) is well defined  only up to a sign. 
Similarly, using the arc $\beta_E^-$ we get a number $\Delta_{\beta_E^-}(v_x, v_y)\in \C^*$. 
Set 
\be \la{formulaBE}
B_E:= \frac{\Delta_{\beta_E^+}(v_x, v_y)}{\Delta_{\beta_E^-}(v_x, v_y)}.
\ee
Evidently the ratio $B_E$ does not depend on the choice of 
the non-zero vectors $v_x$ and $v_y$. So we get a rational function
(\ref{RFBj}).

Changing the orientation of an edge $E$ results in inversion of  both $M_E$ and $B_E$. 

The functions $\{M_E, B_E\}$ do not 
define a rational coordinate system on the space ${\rm Loc}_{SL_2, S; \alpha}$ for the following reason. 
 We are going to show that there is a canonical non-trivial action 
of a group ${\rm Hom}({\rm L}_\alpha,\Z/2\Z) =  (\Z/2\Z)^g$ on ${\rm Loc}_{SL_2, S; \alpha}$ which preserves 
the functions $\{M_E, B_E\}$. 

\subsection{Complex analogs of Fenchel-Nielsen coordinates for arbitrary $G$.} 
Recall that orientations of the edges $E$ provide orientations of the loops $\alpha_E$. 
Using these orientations, 
the semi-simple parts of the monodromies along the loops $\alpha_E$ provide a map
$$
\{M_E\}: {\rm Loc}_{G, S; \alpha} \lra H^{{\cal E}_\Gamma} \stackrel{\sim}{=} H^{3g-3}.
$$
Let ${\rm Conf}^\times_2({\cal A})\subset {\rm Conf}_2({\cal A})$ be the subspace parametrization pairs of decorated flags in generic position. There is a canonical isomorphism
$$
\Delta: {\rm Conf}^\times_2({\cal A}) \stackrel{\sim}{\lra} H. 
$$
Using this, and generalizing the set up of formula (\ref{formulaBE}) from $SL_2$ to arbitrary group $G$ 
by replacing the vectors $v_x, v_y$ there by arbitrary decorated flags $A_x, A_y$ 
at the points $x,y$ which lift the framings over these points, 
 the loop $\beta_E$ provides an  $H$-invariant:
$$
B_E:= \frac{\Delta_{\beta_E^+}(A_x, A_y)}{\Delta_{\beta_E^-}(A_x, A_y)}\in H.
$$
So we get a map 
$$
\{B_E\}: {\rm Loc}_{G, S; \alpha} \lra H^{{\cal E}_\Gamma} \stackrel{\sim}{=} H^{3g-3}.
$$
Changing the orientation of an edge $E$ results in the inversion $h \lms h^{-1}$ of  both $M_E$ and $B_E$. 

\paragraph{The space ${\cal L}_{G, S; \alpha}$ on which the coordinates live.}

Let us formulate first our results. Let $G$ be any split semi-simple algebraic group.

Recall the Lagrangian sublattice ${\rm L}_\alpha\subset H_1(S, \Z)$ generated by the 
loops $\{\alpha_E\}$.

Consider the following finite abelian group: 
\be \la{gracteff}
{\rm Hom}({\rm L}_\alpha,  {\rm Cent}(G)) \stackrel{\sim}{=}{\rm Cent}(G)^g.
\ee
\bp \la{9.18.14.1}
The group ${\rm Hom}({\rm L}_\alpha,  {\rm Cent}(G))$ acts effectively at the generic point of the coarse moduli space ${\rm Loc}_{G, S; \alpha}$.
\ep

\bd \la{9.19.14.1} The space ${\cal L}_{G, S; \alpha}$ is the quotient of 
${\rm Loc}_{G, S; \alpha}$ by the action of the group (\ref{gracteff}). 
\ed

So it is related to the original moduli space ${\rm Loc}_{G, S}$ via the following  diagram, where 
$c_G$ is the order of the center of $G$, and the numbers at the vertical arrows are their degrees: 
$$
\begin{array}{ccccc}
&&{\rm Loc}_{G, S; \alpha}&&\\
&&&&\\
&c_G^{3g-3}\swarrow &&\searrow c_G^{g}&\\
&&&&\\
{\rm Loc}_{G, S}&&&&{\cal L}_{G, S; \alpha}
\end{array}
$$
Notice that the space of $G$-local systems on $S$ in general 
is not a rational variety, i.e. it is not birationally isomorphic to a projective space. 
 So it can not have a rational coordinate system since the latter, by definition, provides 
a birational isomorphism with a projective space. 

\bt \la{9.18.14.2}
a) The functions $\{M_E, B_E\}$ 
define a rational coordinate system on  ${\cal L}_{SL_2, S; \alpha}$. 

b) The  space ${\cal L}_{G, S; \alpha}$ 
is rational. The functions $\{M_E, B_E\}$ are a part of a rational coordinate system on  ${\cal L}_{G, S; \alpha}$. 

\et

Below we prove simultaneously Proposition \ref{9.18.14.1} and Theorem \ref{9.18.14.2}.  
\begin{proof} 

Denote by ${\cal M}_{G, {\cal P}_v}$ the coarse 
moduli space of twisted 
$G$-local systems on a pair of pants 
${\cal P}_v$ equipped with framings at the three boundary loops. 
We denote by ${\cal L}_v$ a point of  ${\cal M}_{G, {\cal P}_v}$. 
Consider the subspace
$$
{\cal M}_{G, \Gamma} \subset \prod_{v \in \Gamma}{\cal M}_{G, {\cal P}_v}
$$
defined by the condition that for each loop $\alpha_E$ the monodromies  of the local systems 
${\cal L}_{v_E^+}$ and ${\cal L}_{v_E^-}$ around $\alpha_E$ coincide. 
There is a surjective restriction map 
$$
{\rm Res}: {\rm Loc}_{G, S; \alpha} \lra {\cal M}_{G, \Gamma}. 
$$

The automorphism group of a generic framed $G$-local system on a space
 with non-abelian fundamental group is the center ${\rm Cent}(G)$ of the group $G$. 

There is a canonical map 
$$
{\rm Cent}(G)^{{\cal V}_\Gamma} \lra {\rm Cent}(G)^{{\cal E}_\Gamma}.  
$$
It assigns to a collection of central element $\{c_v\}$ at the vertices $v$ a collection 
of central elements $\{c_{{E}}\}$ at the oriented edges ${{E}}$ where $c_{{E}}:= c_{s({{E}})}/c_{t({{E}})}$, where 
$s({{E}})$ is the source of the arrow ${{E}}$, 
and $t({{E}})$ is its target. Then, since ${\rm Cent}(G)\subset {\rm H}$, one has 
\be \la{image}
{\rm Im}\Bigl({\rm Cent}(G)^{{\cal V}_\Gamma} \lra {\rm Cent}(G)^{{\cal E}_\Gamma}\Bigr) \subset {\rm H}^{{\cal E}_\Gamma}.
\ee

\bl \la{rationalf}
The group 
$$
{\rm H}^{{\cal E}_\Gamma}/{\rm Im}\Bigl({\rm Cent}(G)^{{\cal V}_\Gamma} \lra {\rm Cent}(G)^{{\cal E}_\Gamma}\Bigr)
$$
acts simply transitively on the fiber of the map ${\rm Res}$ over a generic point of 
${\cal M}_{G, \Gamma}$.
\el

\begin{proof}
Given a collection of framed twisted $G$-local systems $\{{\cal L}_{v}\}$ on pairs of pants ${\cal P}_{v}$ 
whose monodromies around all loops $\alpha_E$ coincide, and given any collection of isomorphisms 
\be \la{coliso}
\{i_E\} \in \prod_{E \in {\cal E}_\Gamma}{\rm Isom}({{\cal L}_{s(E)}}_{|\alpha_E} \lra {{\cal L}_{t(E)}}_{|\alpha_E}), 
\ee
one can glue a  twisted 
 $G$-local system ${\cal L}$ on $S$ with a framing on the  $\alpha$. 
So a gluing data $(\{{\cal L}_{v}\}, \{i_E\} )$ determines 
uniquely such an ${\cal L}$ which restricts to the collection $\{{\cal L}_v\}$. 
Let us find out when two gluing data $(\{{\cal L}_{v}\}, \{i_E\} )$ and $(\{{\cal L}'_{v}\}, \{i'_E\} )$ 
determine isomorphic   ${\cal L}$'s on $S$.

The automorphism group of a generic twisted framed $G$-local system 
on a circle is the Cartan group $H$ of $G$ - the centralizer of a generic element of $G$. 
Therefore for a generic $G$-local system, the
 group ${\rm H}^{{\cal E}_\Gamma}$ acts simply transitively on the space of gluing isomorphisms (\ref{coliso}). 
The group ${\rm Cent}(G)$ acts by automorpisms of ${\cal L}_{v}$ for each vertex $v$ of $\Gamma$. 
So the group ${\rm Cent}(G)^{{\cal V}_\Gamma}$ acts by automorphisms of the collection 
$\{{\cal L}_{v}\}$. It does not change the isomorphism classes 
of the ${\cal L}_{v}$'s, but does change the collection of isomorphisms (\ref{coliso}). 
Evidently the group ${\rm Cent}(G)^{{\cal V}_\Gamma}$ acts on isomorphisms (\ref{coliso}) 
via its image in ${\rm Cent}(G)^{{\cal E}_\Gamma}$. For generic ${\cal L}_v$ one has 
$$
{\rm Aut}({\cal L}_v) = {\rm Cent}(G).
$$
So the 
 isomorphism classes of the glued ${\cal L}$'s on $S$ are the orbits of  the group (\ref{image}). 
\end{proof}

Now we can finish the proof of Proposition \ref{9.18.14.1}. 
Indeed, it is clear from (\ref{Lagiso}) that one has 
$$
{\rm Cent}(G)^{{\cal E}_\Gamma}/{\rm Im}\Bigl({\rm Cent}(G)^{{\cal V}_\Gamma} \lra {\rm Cent}(G)^{{\cal E}_\Gamma}\Bigr) 
{=}{\rm Hom}({\rm L}_\alpha,  {\rm Cent}(G)).
$$

Let us prove now Theorem \ref{9.18.14.2}.

a) Recall the restriction  map
$
{\rm Res}: {\cal L}_{G, S; \alpha} \lra {\cal M}_{G, \Gamma}.
$
 
If $G=SL_2$, then ${\cal M}_{G, {\cal P}_v} = {\Bbb G}_m^3$, 
with the monodromies around the three boundary loops providing the isomorphism. The fibers over the generic points are rational by Lemma \ref{rationalf}. 
So the total space ${\cal L}_{SL_2, S; \alpha}$ is rational.  Moreover, 
the restriction map for $SL_2$ boils down to  the monodromies of the twisted framed $SL_2$-local systems 
over the loops of the lamination $\alpha$:
$$
{\rm Res} = M_\alpha: {\rm Loc}_{SL_2, S; \alpha} \lra ({\Bbb G}_m)^{3g-3}.
$$
Rescaling a component $i_E$ of the  gluing data by $\lambda$ rescales 
$B_E$ by $\lambda^2$, and leaves untouched the other $B$-coordinates. 
Therefore given a generic fiber  
$M_\alpha^{-1}(x)$, the functions $\{B_E\}$ provide its isomorphism 
with a torus
$$
\{B_E\}: {M_\alpha}^{-1}(x)\lra ({\Bbb G}_m/\pm 1)^{3g-3}.
$$
Definition \ref{9.19.14.1} of the space ${\cal L}_{SL_2, S; \alpha}$ 
kills the action of the group $(\pm 1)^{3g-3}$ on the isomorphisms $i_E$. 
So the functions $(M_E, B_E)$ separate generic points, providing 
a birational isomorphism
$$
(M_E, B_E): {\rm Loc}_{SL_2, S; \alpha} \stackrel{\sim}{\lra} ({\Bbb G}_m)^{6g-6}.
$$

b) Given a pair of pants ${\cal P}$, the moduli space ${\cal M}_{G, {\cal P}}$ 
of framed twisted $G$-local systems on 
${\cal P}$ has a positive structure. In the case when ${\rm Cent}(G)$ is trivial this was 
proved in \cite{FG1}. The monodromy around the three boundary loops 
provides a positive map 
$$
\mu_{\cal P}: {\cal M}_{G, {\cal P}} \lra H^3.
$$
Its fiber ${\cal M}^{\rm un}_{G, {\cal P}}$ over the unit element ${\cal M}^{\rm un}_{G, {\cal P}}$ 
parametrizes the subspace of unipotent framed $G$-local systems. It is a positive space. 
In particular it is rational. 
One can split non-canonically the map $\mu_{\cal P}$, getting a positive projection 
$\nu_{\cal P}: {\cal M}_{G, {\cal P}} \lra {\cal M}^{\rm un}_{G, {\cal P}}$ and 
therefore a positive birational isomorphism 
$$
(\nu_{\cal P}, \mu_{\cal P}): {\cal M}_{G, {\cal P}} \lra {\cal M}^{\rm un}_{G, {\cal P}} \times H^3.
$$
Therefore there is a birational isomorphism 
\be \la{Decomap}
(\{\nu_{{\cal P}_v}\}, \{M_E, B_E\}): {\rm Loc}_{G, S; \alpha} 
\lra \prod_{v \in {\cal V}_\Gamma}{\cal M}^{\rm un}_{G, {\cal P}_v} 
\times (H\times H)^{{\cal E}_\Gamma}.
\ee

Then the arguments are just as in the $SL_2$ case. 
\end{proof}

\section{Appendix: The quantum cluster symplectic double} \la{Sec6}

\begin{definition}
A {\it quiver}  is a datum 
$$
{\bf q} = \Bigl(\Lambda, \{e_i\}, (\ast, \ast)\Bigr). 
$$
 Here 
$\Lambda$ is a  lattice, $\{e_i\}$ is its basis, 
and $(\ast, \ast)$ a skew-symmetric $\Z$-valued 
bilinear form on $\Lambda$.

A mutation of a quiver ${\bf q}$ 
in the direction of a basis vector $e_k$ is a new quiver $$
\widetilde {\bf q} = \Bigl(\Lambda, \{\widetilde e_i\}, (\ast, \ast)\Bigr).
$$
It has 
 the same  lattice and form as the original quiver ${\bf q}$, and a new basis $\{\widetilde e_i\}$  defined by
\begin{equation} \label{12.12.04.2a}
\widetilde e_i := 
\left\{ \begin{array}{lll} e_i + (e_{i}, e_k)_+e_k
& \mbox{ if } &  i\not = k, ~~ (\alpha)_+= {\rm max}(0, \alpha);\\
-e_k& \mbox{ if } &  i = k.\end{array}\right.
\end{equation}
\end{definition}

Consider the double $\Lambda_{\cal D}$ of the lattice $\Lambda$:
$$
\Lambda_{\cal D}:= \Lambda \oplus \Lambda^\vee, \qquad \Lambda^\vee:= {\rm Hom}(\Lambda, \Z).
$$
It gives rise to a split algebraic torus 
$$
{\mathcal T}_{\Lambda}:= {\rm Hom}(\Lambda_{\cal D}, \C^*).
$$
 
The basis $\{e_i\}$ of $\Lambda$
 provides the dual basis $\{e^{\vee}_j\}$ of $\Lambda^\vee$. So a quiver ${\bf q}$ provides a basis 
\be \la{12.4.12.1}
\mbox{$\{e_i, e^{\vee}_j\}$ of $\Lambda_{\cal D}$}.
\ee
The basis (\ref{12.4.12.1}) of the lattice  $\Lambda_{\cal D}$ gives rise to 
 the coordinates $\{X_i, B_j\}$ of the torus ${\mathcal T}_{\Lambda}$.

\vskip 3mm
The lattice 
$\Lambda_{\cal D}$ with the form $(\ast, \ast)_{\cal D}$ gives rise to the quantum torus algebra 
${\bf T}$. Precisely, the basis (\ref{12.4.12.1}) gives rise to a set of the ``quantum coordinates'' 
$(X_i, B_j)$ -- generators of the quantum torus algebra 
${\bf T}$ -- satisfying    
the relations 
$$
B_iB_j=B_jB_i, \quad q^{-1} X_iB_i  = 
q B_iX_i, \quad  B_i X_j = 
X_jB_i, \mbox{~$i \not = j$}, \quad  
q^{-(e_i, e_j)} X_i X_j = 
q^{-(e_j, e_i)} X_jX_i.
$$

Denote by $ {\Bbb T}$ the (non-commutative) fraction field of ${\bf T}$. 
Recall the quantum dilogarithm power series, although known as the quantum exponential:
$$
{\bf \Psi}_q(x)= 
\prod_{k=1}^{\infty} (1+q^{2k-1}x)^{-1}. 
$$ 

Let us introduce the following notation
$$
{\Bbb B}_k^+:= \prod_{i~|~(e_k, e_i)>0}B_i^{(e_k, e_i)}, \qquad
{\Bbb B}_k^-:= \prod_{i~|~(e_k, e_i)<0}B_i^{-(e_k, e_i)}.
$$
\begin{equation}\label{zx1q1}
X^o_i:= X_i ~ \frac{{\Bbb B}_k^+}{{\Bbb B}_k^-} = X_i
\prod_{j\in I}B_j^{(e_i, e_j)}. 
\end{equation}
Notice that all variables which appear in the definition of $X^o_i$ commute. 

\begin{theorem-definition} \label{4.28.03.11ccc}
The conjugation by 
${\bf \Psi}_{q}(X_k)/{\bf \Psi}_{q}(X^o_k)$ provides an automorphism 
\be \la{4.14.12.1}
\mu^*_{e_k}: {\Bbb T} \lra {\Bbb T}.
\ee 
\end{theorem-definition}
Notice that, although ${\bf \Psi}_{q}(X_k)$ is a power series, we get 
a birational automorphism. 

\paragraph{Mutations on the classical level.} Let us apply the quantum automorphism (\ref{4.14.12.1}) to the generators $\{\widetilde X_i, \widetilde B_j\}$ 
assigned  to the mutated basis 
$(\widetilde e_i, \widetilde e^\vee_j)$, express the result  via 
the generators $\{X_i, B_j\}$  
assigned finally to the original basis $(e_i, e^\vee_j)$, and set $q=1$. 
Then we calculate the obtained birational transformation of the 
torus ${\mathcal T}_{\Lambda}$:
\be \la{4.14.12.2}
\mu^*_{e_k}: \Q({\mathcal T}_{\Lambda}) \lra \Q({\mathcal T}_{\Lambda}).
\ee

\bt The action of the birational automorphism (\ref{4.14.12.2}) on  the coordinates 
$\{\widetilde X_i, \widetilde B_j\}$, expressed in terms of the coordinates  $\{X_i, B_j\}$ is given by 
\begin{equation}\label{zx1qt}
\mu^*_{e_k}: \widetilde X_{i} \lms \left\{\begin{array}{lll} X_k^{-1}& \mbox{ if $i=k$}  \\
 X_i(1+X_k^{-\sgn(\varepsilon_{ik})})^{-\varepsilon_{ik}}, & \mbox{ if $i \not = k$}.
\end{array} \right.
\end{equation}
\begin{equation} \label{4.28.03.11x}
\mu^*_{e_k}: \widetilde B_k \lms 
\frac{{\Bbb B}_k^- + X_k{\Bbb B}_k^+}{B_k(1+X_k)}, \qquad \widetilde B_j \lms B_j~~\mbox{ if $j\not =k$}.
\end{equation}
\et

There is  a symplectic form on the torus ${\cal T}_{\Lambda}$, given in coordinate $(X_i, B_j)$ by 
\begin{equation}\label{ELLs}
\Omega_{\bf q}= -\frac{1}{2}\sum_{i,j} 
(e_i, e_j)\cdot  d\log B_i \wedge d\log B_j - \sum_{i}  d\log B_i \wedge d\log X_i.
  \end{equation}
The Poisson structure provided by the symplectic form is given in coordinates $\{X_i, B_j\}$ by 
\begin{equation}\label{zx1}
\{B_i, B_j\}= 0, \quad \{X_i, B_j\}= \delta_{ij}X_iB_j, \quad \{X_i, X_j\} = 
\varepsilon_{ij}X_iX_j.
\end{equation}

The symplectic form is obtained by applying the $d\log \wedge d\log$ map to a class 
\begin{equation}\label{ELL}
W_{\bf q}= -\frac{1}{2}\sum_{i,j} 
(e_i, e_j)\cdot  B_i \wedge B_j - \sum_{i}  B_i \wedge X_i
  \in \Lambda^2\Q({\cal T}_{\Lambda})^*. 
  \end{equation}

\bt i) Given a mutation ${\bf q} \to \widetilde {\bf q}$ in the direction $e_k$, one has 
$$
\mu_{e_k}^*W_{\bf \widetilde q}  - W_{\bf q} = 
(1+ X^o_k) \wedge X^o_k - 
(1+X_k) \wedge X_k.
$$ 
ii) The mutations preserve the symplectic, and hence the Poisson structure. 
\et

The second claim follows immediately from the first. The first is Proposition 2.14 in \cite{FG3}. 
The first claim implies that there is a canonical class in $K_2({\cal T}_{\Lambda})$ preserved by the mutations. 
So it gives rise to a canonical line bundle with connection on the symplectic double. 
Notice that the part ii) follows immediately from Theorem-Definition \ref{4.28.03.11ccc}. 
Indeed, the classical limit of an automorphism of a quantum torus algebra preserves the 
corresponding Poisson structure.

\paragraph{The classical 
cluster symplectic double.} Now we are ready to define the cluster symplectic double variety. 
The construction follows the definition of cluster Poisson and $K_2$-varieties given in \cite{FG2}. 
Starting with a quiver ${\bf q}$, we assign to it the split algebraic torus 
${\cal T}_{\bf q}$ with the cluster symplectic double coordinates $(B_i, X_i)$. 
Then we mutate the quiver ${\bf q}$ in the directions of all basis vectors, getting 
new split algebraic tori, and continue this process indefinitely. We glue each pair of 
split tori ${\cal T}_{\bf q}$ and ${\cal T}_{\bf \widetilde q}$ related by a quiver mutation 
according to the mutation formula 
(\ref{zx1qt})-(\ref{4.28.03.11x}). Finally, given two quivers ${\bf q}$ and ${\bf q'}$ 
related by a sequence of mutations, such that there is an isomorphism 
of quivers $i: {\bf q} \to {\bf q'}$ which induces the same 
isomorphism of tori $i_{\cal T}: {\cal T}_{\bf q} \to {\cal T}_{\bf q'}$ as the 
sequence of cluster mutations relating ${\bf q}$ and ${\bf q'}$, we identify the 
tori ${\cal T}_{\bf q}$ and ${\cal T}_{\bf q'}$ according the isomorphism $i_{\cal T}$. 
This way we get a possibly non-separable prescheme, which by abuse of terminology is called 
a cluster symplectic double variety. 

\vskip 3mm
By talking about a {\it cluster symplectic double variety structure on an actual space ${\cal D}$} we 
mean that ${\cal D}$ has a collection of rational coordinate systems,  
assigned to cluster mutations of a quiver 
${\bf q}$ as explained above, and related by the compositions 
of cluster symplectic double transformations (\ref{zx1qt})-(\ref{4.28.03.11x}). 
Precisely,  given a space ${\cal D}$ we have to provide the following: 

\begin{itemize} 

\item A \underline{collection of quivers}, usually infinite, such that any two of them are 
related by a sequence of quiver mutations inside of a given collection.

\item A \underline{construction} assigning to each of the quivers  
a cluster symplectic double rational coordinate system on ${\cal D}$. 

\item A \underline{proof} that the coordinates systems assigned to any pair of the quivers related by a 
 quiver mutation are related by the cluster symplectic double 
transformations (\ref{zx1qt})-(\ref{4.28.03.11x}). 

\end{itemize} 

This is precisely what we do in the paper: define a moduli space
${\cal D}_{\G, \bS}$ assigned to a pair $(\G, \bS)$; consider the collection of 
quivers introduced in \cite{FG1} for the pair $(PGL_m, \bS)$; construct 
the cluster symplectic double rational coordinate systems on the moduli space 
${\cal D}_{PGL_m, \bS}$, and prove that they are related by the cluster symplectic double 
transformations (\ref{zx1qt})-(\ref{4.28.03.11x}).

\end{document}